\documentclass[11pt]{amsart}
\usepackage[dvips]{color}
\usepackage{amsmath}
\usepackage{amsxtra}
\usepackage{amscd}
\usepackage{amsthm}
\usepackage{amsfonts}
\usepackage{amssymb}
\usepackage{eucal}
\usepackage{epsfig}
\usepackage{graphics}
\theoremstyle{plain}
\newtheorem{theorem}{Theorem}[section]

\newtheorem{thm}[theorem]{Theorem}
\newtheorem{cor}[theorem]{Corollary}
\newtheorem{prop}[theorem]{Proposition}
\newtheorem{lem}[theorem]{Lemma}

\newtheorem{defi}[theorem]{Definition}
\theoremstyle{definition}

\newcommand{\Glie}{\mathfrak{g}}
\newcommand{\Yim}{\mathcal{Y}}

\newcommand{\U}{\mathcal{U}}

\newcommand{\ZZ}{\mathbb{Z}}
\newcommand{\CC}{\mathbb{C}}

\newcommand{\cN}{\mathcal{N}}

\newcommand{\C}{\mathbb{C}}

\newcommand{\Z}{\mathbb{Z}}

\newcommand{\g}{\mathfrak{g}}
\newcommand{\bo}{\mathfrak{b}}
\newcommand{\tb}{\mathbf{\mathfrak{t}}}
\newcommand{\ga}{\overline{\alpha}}
\newcommand{\gb}{\dot{\mathfrak{g}}}
\newcommand{\Utg}{\widetilde{U}_q(\mathfrak{g})}
\newcommand{\Psib}{\mbox{\boldmath$\Psi$}}
\newcommand{\Psibs}{\scalebox{.7}{\boldmath$\Psi$}}
\newcommand{\qbin}[2]{{\left[
\begin{matrix}{\,\displaystyle #1\,}\\
{\,\displaystyle #2\,}\end{matrix}
\right]
}}

\newtheorem{rem}[theorem]{Remark}

\begin{document}

\begin{title}[Asymptotic representations]
{Asymptotic representations \\and Drinfeld rational fractions}
\end{title}
\author{David Hernandez and Michio Jimbo}
\address{DH: Institut de Math{\'e}matiques de Jussieu,
Universit{\'e} Paris Diderot (Paris VII), 175 rue du Chevaleret 75013 Paris, France}
\email{hernandez@math.jussieu.fr}
\address{MJ: Department of Mathematics, 
Rikkyo University, Toshima-ku, Tokyo 171-8501, Japan}
\email{jimbomm@rikkyo.ac.jp}

\begin{abstract}
We introduce and study a category of representations of 
the Borel algebra, associated with a quantum loop algebra of non-twisted type.
We construct fundamental representations for this category as a limit
of the Kirillov-Reshetikhin modules over the quantum loop algebra 
and we establish explicit formulas for their characters. 
We prove that general simple modules in this category 
are classified by $n$-tuples 
of rational functions in one variable, which are regular and non-zero at the origin but 
may have a zero or a pole at infinity.

\vskip 4.5mm



\end{abstract}

\maketitle


\section{Introduction}\label{Intro}

The theory of finite dimensional representations of quantum loop algebras
is a rich subject, developed already by many authors. 
It has many applications and connections to various branches of mathematics, 
including representation theory, algebraic geometry, combinatorics and 
classical/quantum integrable systems. 
For a recent review the reader is referred 
to \cite{ch, le}.

Let us briefly recall the classification of 
irreducible finite dimensional representations. 
Let $U_q(\mathfrak{g})$ be a quantum loop algebra of non-twisted type, 
and let $x^{\pm}_i(z)$, $\phi^{\pm}_i(z)$ ($1\le i\le n$) be the 
Drinfeld currents (for the notation, see Section \ref{bck} below).
Then any irreducible (type 1) finite dimensional 
$U_q(\mathfrak{g})$-module $V$ is 
presented as $V=U_q(\mathfrak{g})\, v$, with $v$ a non-zero vector satisfying
\begin{align*}
x^+_i(z)\,v=0,\quad \phi^\pm_i(z)\,v=\Psi_i(z)\,v\quad (i=1,\cdots,n)\,.
\end{align*}
The eigenvalues of $\phi^\pm_i(z)$ are the power series expansions in 
$z^{\pm 1}$ of a rational function $\Psi_i(z)$ of a specific type, 
\begin{align*}
\Psi_i(z)=q_i^{\mathrm{deg}\,P_i}\frac{P_i(q_i^{-1}z)}{P_i(q_iz)}\,,
\end{align*}
where $P_i(z)\in \C[z]$ is a  polynomial 
such that $P_i(0)=1$.
These are called Drinfeld polynomials. 
The $n$-tuple $\Psib=\bigl(\Psi_i\bigr)_{i=1,\cdots,n}$ 
is called the highest $\ell$-weight of $V$. 
Note that each $\Psi_i(z)$ is regular and non-zero at $z=0,\infty$.
The correspondence between $V$ and 
$\bigl(P_i(z)\bigr)_{1\le i\le n}$ is bijective 
\cite{Cha1, Cha2}. 

A particularly well-studied case is the Kirillov-Reshetikhin (KR) 
module defined by 
\begin{align*}
P_i(z)=\prod_{l=1}^k(1-a q_i^{k-2l+1}z)\,,
\quad P_j(z)=1\ (j\neq i)\,
\end{align*}
for some $i=1,\cdots,n$, $k\in\Z_{\ge0}$ and $a\in\C^*$. 
In this article we denote it by $L(M^{(i)}_{k,a})$.
The KR modules have various nice properties. Among others, 
it is known \cite{n, hcr} that as $k\to\infty$ 
the normalized $q$-character 
$\tilde{\chi}_q\bigl(L(M^{(i)}_{k,q_i^{-2k+1}})\bigr)$
\cite{Fre}
has a well-defined limit as a formal power series. 

In view of these results, one is naturally led to ask the following
questions:
\begin{enumerate}
\item In the description of highest $\ell$-weight modules, 
what will happen if $\Psi_i(z)$'s are more general rational fractions? 
\item What is the representation theoretical content of 
$\lim_{k\to\infty}\tilde{\chi}_q\bigl(L(M^{(i)}_{k,q_i^{-2k+1}})\bigr)$?
\end{enumerate}

Motivated by these questions, we introduce and study in this article 
a certain category $\mathcal{O}$ of $U_q(\mathfrak{b})$-modules,  
where $U_q(\mathfrak{b})$ is the Borel subalgebra of $U_q(\mathfrak{g})$. 
Basic notions for $U_q(\mathfrak{g})$ such as highest $\ell$-weight 
carry over in a straightforward manner to $U_q(\mathfrak{b})$ as well.  
We prove that a simple highest $\ell$-weight $U_q(\mathfrak{b})$-module belongs to category $\mathcal{O}$ 
if and only if its highest $\ell$-weight $\Psib=(\Psi_i(z))_{1\le i\le n}$ consists of rational functions 
which are regular and non-zero at $z=0$. We call them the Drinfeld rational fractions.  
Since no condition at $z=\infty$ is required, 
this gives an answer to question (i).
These simple modules are infinite dimensional in general,
but we prove that their weight spaces are finite dimensional.

Among simple highest $\ell$-weight modules,  
of particular significance 
are the two kinds of fundamental modules $L^\pm_{i,a}$ 
($1\leq i\leq n$, $a\in\CC^*$)
defined by the highest $\ell$-weight
\begin{align*}
\Psi_i(z)=(1-az)^{\pm 1}\,,
\quad \Psi_j(z)=1\ (j\neq i)\,.
\end{align*}
General simple highest $\ell$-weight modules are subquotients of tensor products 
of $L^\pm_{i,a}$'s. 
For the proof of the classification mentioned above, 
the key point lies in 
showing that $L^\pm_{i,a}$'s belong to category $\mathcal{O}$. 
We do this by constructing the module $L^-_{i,a}$
as the limit $k\to\infty$ of KR modules 
$M^{(i)}_{i,q_i^{-2k+1}a}$ viewed as $U_q(\mathfrak{b})$-modules 
(with an appropriate shift of grading). 
The module $L^+_{i,a}$ are then constructed 
either as a similar limit of $M^{(i)}_{i,q_i^{2k-1}a}$ or using 
a duality argument.
In particular the limit of the normalized $q$-characters 
of $M^{(i)}_{i,q_i^{-2k+1}a}$ is the $q$-character of 
the $U_q(\mathfrak{b})$-module $L^-_{i,1}$. 
This gives an answer to question (ii).   
These are the main results of the present paper.

Actually we prove that the module  $L^-_{i,a}$  admits an action of 
a `larger' algebra $\widetilde{U}_q(\mathfrak{g})$ 
which we call asymptotic algebra. 
It is defined in the same way as the quantum loop algebra 
$U_q(\mathfrak{g})$ wherein invertibility of the $\phi^-_i(0)$'s
is not assumed. 
Asymptotic algebra does not contain 
$U_q(\mathfrak{b})$ as a subalgebra, but it
is `larger' in the sense that 
$Q$-graded $\widetilde{U}_q(\mathfrak{g})$-modules 
($Q$ denoting the root lattice of the classical part of 
$\mathfrak{g}$) can be regarded uniquely 
as $U_q(\mathfrak{b})$-modules.
See Proposition \ref{asym} for the precise statement. 
In contrast, the action of $U_q(\mathfrak{b})$ on 
$L^+_{i,a}$ cannot be extended to that of the 
full asymptotic algebra $\widetilde{U}_q(\mathfrak{g})$. 

Since explicit character formulas for KR modules are known 
\cite{n,hcr}, 
our asymptotic construction of the fundamental representations 
$L^\pm_{i,a}$ implies
explicit formulas for their characters $\chi(L^\pm_{i,a})$. 
Moreover, our results
imply that $\chi(L_{i,a}^+) = \chi(L_{i,b}^-)$ for any 
$i=1,\cdots,n$ and $a,b\in\CC^*$.

Some particular cases of fundamental
representations $L^\pm_{i,a}$ of the Borel algebra
have been discussed
in the literature 
under the name of `$q$-oscillator representations', 
see e.g.
\cite{BLZ}, Appendix D for $\widehat{\mathfrak{sl}}_2$, 
\cite{BHK}, Appendix B for $\widehat{\mathfrak{sl}}_3$, 
\cite{BTs}, Section 5  for $\widehat{\mathfrak{sl}}(2|1)$,
and  
\cite{Ko} for $L_{1,a}^\pm$ of $\widehat{\mathfrak{sl}}_n$. 
The corresponding $q$-characters are intimately connected with 
Baxter's $Q$-operators important in quantum integrable systems. 
Giving a systematic account to this construction
has been another motivation for our study. 

As mentioned above, the theory of finite dimensional 
representations of quantum affine
algebras has been intensively studied. 
We can expect similar developments for 
the category $\mathcal{O}$ of the Borel algebra. 
For example, it would be very interesting to find 
the defining relations for its Grothendieck ring. 

The standard asymptotic representation theory \cite{v} involves limits
with respect to the rank of a Lie algebra or a group. 
In the present paper, the limits 
are taken with respect to the level of representations. 
Here, by `level' we mean that of 
representations of the quantum group associated 
to the underlying finite-dimensional Lie algebra.
The relation between these two 
kinds of asymptotic representation theory should be understood 
in the spirit of the
level/rank duality of \cite{f}.

Another direction is to understand the connection of our work
to results about Demazure modules for classical current algebras, 
such as in \cite{fl}.

We hope to return to these issues in the near future.

The text is organized as follows.

In section \ref{bck} we set up the notation concerning $U_q(\mathfrak{g})$, 
introduce the asymptotic algebra  $\widetilde{U}_q(\mathfrak{g})$
and discuss the connection between representations of 
$\widetilde{U}_q(\mathfrak{g})$ and that of 
the Borel algebra $U_q(\mathfrak{b})$.  
In section \ref{catO} we discuss the basics of 
highest $\ell$-weight modules of $U_q(\mathfrak{b})$.  
We introduce the category $\mathcal{O}$,  
and state the classification of simple modules
in Theorem \ref{class}, whose proof  
will be given in the following sections. 
We also discuss the $q$-characters 
and summarize some facts about 
finite dimensional representations of $U_q(\mathfrak{g})$,  
including the results 
from \cite{h3} which will be used in subsequent sections. 
We analyze analogous properties for a category $\mathcal{O}^*$ 
dual to $\mathcal{O}$.
In section \ref{asymptminus} we show that the family of 
KR modules $\{M^{(i)}_{k,q_i^{-2k+1}}\}_{k\ge0}$  
has a well defined limit $V_\infty$, and that it has 
the structure of a module over the asymptotic algebra
$\widetilde{U}_q(\mathfrak{g})$. 
As a consequence, the simple module $L^-_{i,1}$ is shown to belong to
category $\mathcal{O}$. 

The next section concerns the simple module  $L^+_{i,1}$.  
In section \ref{asymptplus} we consider a dual
module $(V_\infty^\sigma)^*$ of the module $V_\infty$ constructed
in the previous section, over the algebra 
$\widetilde{U}_{q^{-1}}(\mathfrak{g})$. From this we conclude 
that  $L^+_{i,1}$ is in category $\mathcal{O}$.
By using this result, we prove in section \ref{irred} that 
$V_\infty$ is irreducible, so that it coincides with 
the simple module $L^-_{i,1}$. 
By using duality, we also prove that
$(V_\infty^\sigma)^*$ is irreducible isomorphic to $L_{i,1}^+$.
In particular, we get an explicit character formula for $L^\pm_{i,a}$.
In the last section \ref{asymptpplus} we discuss 
the asymptotic construction of $L^+_{i,1}$.

\section{Quantum loop algebra and related algebras}\label{bck}

In this section,  
we introduce our notation concerning the quantum loop algebra 
and related algebras.

\subsection{Quantum loop algebra}\label{debut}

Let $C=(C_{i,j})_{0\leq i,j\leq n}$ be an indecomposable Cartan matrix 
of non-twisted affine type. 
We denote by $\g$ the Kac-Moody Lie algebra associated with $C$. 
Set $I=\{1,\ldots, n\}$, and 
denote by $\gb$ the finite-dimensional 
simple Lie algebra associated with the Cartan matrix $(C_{i,j})_{i,j\in I}$. 
Let $\{\alpha_i\}_{i\in I}$, $\{\alpha_i^\vee\}_{i\in I}$,
$\{\omega_i\}_{i\in I}$ 
be the simple roots, the simple coroots and the fundamental weights of $\gb$, 
respectively.
We set $Q=\oplus_{i\in I}\Z\alpha_i$, $Q^+=\oplus_{i\in I}\Z_{\ge0}\alpha_i$.
Let $D=\mathrm{diag}(d_0\ldots, d_n)$ be the unique diagonal matrix such that 
$B=DC$ is symmetric and $d_i$'s are relatively prime positive integers.
We denote by $(~,~):Q\times Q\to\Z$ the invariant symmetric bilinear 
form such that 
$(\alpha_i,\alpha_i)=2d_i$. Let $a_0,\cdots,a_n$ stand for the Kac 
label (\cite{ka}, pp.55-56).

Throughout this paper, we fix  a non-zero complex number $q$ which is not a root of unity. 
We set $q_i=q^{d_i}$. 

The quantum loop algebra $U_q(\g)$ is the $\C$-algebra defined by generators 
$e_i,\ f_i,\ k_i^{\pm1}$ ($0\le i\le n$) 
and the following relations for $0\le i,j\le n$.
\begin{align*}
&k_ik_j=k_jk_i,\quad k_0^{a_0}k_1^{a_1}\cdots k_n^{a_n}=1,\\
&k_ie_jk_i^{-1}=q_i^{C_{i,j}}e_j,\quad k_if_jk_i^{-1}=q_i^{-C_{i,j}}f_j,
\\
&[e_i,f_j]
=\delta_{i,j}\frac{k_i-k_i^{-1}}{q_i-q_i^{-1}},
\\
&\sum_{r=0}^{1-C_{i.j}}(-1)^re_i^{(1-C_{i,j}-r)}e_j e_i^{(r)}=0\quad (i\neq j),
&\sum_{r=0}^{1-C_{i.j}}(-1)^rf_i^{(1-C_{i,j}-r)}f_j f_i^{(r)}=0\quad (i\neq j)\,.
\end{align*}
Here we have set $x_i^{(r)}=x_i^r/[r]_{q_i}!$ ($x_i=e_i,f_i$), 
and used the standard symbols for $q$-integers 
\begin{align*}
[m]_z=\frac{z^m-z^{-m}}{z-z^{-1}}, \quad
[m]_z!=\prod_{j=1}^m[j]_z,
 \quad 
\qbin{s}{r}_z
=\frac{[s]_z!}{[r]_z![s-r]_z!}. 
\end{align*}
The algebra $U_q(\g)$ has a Hopf algebra structure. We choose the coproduct and the antipode given by
\begin{align*}
&\Delta(e_i)=e_i\otimes 1+k_i\otimes e_i,\quad
\Delta(f_i)=f_i\otimes k_i^{-1}+1\otimes f_i,
\quad
\Delta(k_i)=k_i\otimes k_i\,,
\\
&
S(e_i) = -k_i^{-1} e_i ,\quad S(f_i) = -f_i k_i,
\quad
S(k_i)=k_i^{-1}\,,
\end{align*}
where $i=0,\cdots,n$. 
In particular,  
\begin{align*}
S^{-1}(e_i) = - e_ik_i^{-1} ,\quad
S^{-1}(f_i) = - k_i f_i,
\quad
S^{-1}(k_i)=k_i^{-1}\,.
\end{align*}

\subsection{Drinfeld generators and asymptotic algebras}\label{asy}

The algebra $U_q(\g)$ has another presentation in terms of the Drinfeld generators. 
For our purposes it is convenient to introduce them in the following manner. 

Let $\Utg$ be the $\C$-algebra defined by generators 
\begin{align*}
\tilde{x}_{i,r}^{\pm}\ (i\in I, r\in\Z), 
\quad 
\tilde{\phi}_{i,\pm m}^\pm\ (i\in I, m\geq 0),
\quad \kappa_i\ (i\in I) 
\end{align*}
and the following defining relations for all $i,j\in I$, $r,r'\in\Z$, $m,m'\geq 0$: 

\begin{align*}
&\tilde{\phi}_{i,0}^+= 1\,,\tilde{\phi}_{i,0}^-= \kappa_i^2\,,
\quad
[\tilde{\phi}_{i,\pm m}^\pm,\tilde{\phi}_{j,\pm m'}^\pm] = 0,\ 
[\tilde{\phi}_{i,\pm m}^\pm,\tilde{\phi}_{j,\mp m'}^\mp] = 0\,,
\\
&\kappa_i\tilde{x}^\pm_{j,r}=q_i^{\mp C_{i,j}}\tilde{x}^{\pm }_{j,r}\kappa_i,
\\
&
\tilde{\phi}_{i,m}^+\tilde{x}_{j,r}^{\pm}
= \sum_{0\leq l\leq m}
q_i^{\pm l C_{i,j}}\tilde{x}_{j,r+l}^\pm\tilde{\phi}_{i,m-l}^+
- \sum_{0\leq l\leq m-1}q_i^{\pm (l-1) C_{i,j}}\tilde{x}_{j,r+l+1}^\pm
\tilde{\phi}_{i,m-l-1}^+\,,
\\
&
\tilde{\phi}_{i,-m}^- \tilde{x}_{j,r}^{\pm} 
= -\sum_{0\leq l\leq m - 1}
q_i^{\mp (l+1)C_{i,j}}\tilde{x}_{j,r-l-1}^{\pm} \tilde{\phi}_{i,-m+l+1}^-
+ \sum_{0\leq l\leq m}q_i^{\mp (l+2) C_{i,j}} 
\tilde{x}_{j,r-l}^{\pm}\tilde{\phi}_{i,-m+l}^-\,,
\\
&
q_i^{-C_{i,j}}\tilde{x}_{i,r}^+\tilde{x}_{j,r'}^-
-\tilde{x}_{j,r'}^-\tilde{x}_{i,r}^+
= 
\delta_{i,j}\frac{\tilde{\phi}^+_{i,r+r'}-\tilde{\phi}^-_{i,r+r'}}{q_i-q_i^{-1}}\,,
\\
&
\tilde{x}_{i,r+1}^{\pm}\tilde{x}_{j,r'}^{\pm}
-q_i^{\pm C_{i,j}}\tilde{x}_{j,r'}^{\pm}\tilde{x}_{i,r+1}^{\pm}
=q_i^{\pm C_{i,j}}\tilde{x}_{i,r}^{\pm}\tilde{x}_{j,r'+1}^{\pm}
-\tilde{x}_{j,r'+1}^{\pm}\tilde{x}_{i,r}^{\pm}\,,
\\
&
\sum_{\pi\in \Sigma_s}
\sum_{k=0}^{s}
(-1)^k
\qbin{s}{k}_{q_i}
\tilde{x}_{i,r_{\pi(1)}}^{\pm}\cdots 
\tilde{x}_{i,r_{\pi(k)}}^{\pm}
\tilde{x}_{j,r'}^{\pm}
\tilde{x}_{i,r_{\pi(k+1)}}^{\pm}\cdots \tilde{x}_{i,r_{\pi(s)}}^{\pm}
=0\,. 
\end{align*}
In the last relation, $i\neq j$, $s=1-C_{i,j}$,  
$ r_1,\cdots,r_s$ run over all integers, and 
$\Sigma_s$ stands for the symmetric group on $s$ letters.  
We have set also $\tilde{\phi}_{i,\pm m}^\pm = 0$ for $m < 0$.
We call  $\Utg$ the asymptotic algebra. 
Notice that the elements $\kappa_i\in \Utg$ are not assumed to be invertible. 
Indeed, we shall consider later representations of $\Utg$ on which $\kappa_i$ act as $0$. 

We have then 
\begin{prop}\cite{Dri2,bec}
There is an isomorphism of $\C$-algebras 
\begin{align}
U_q(\g)\simeq \Utg\otimes_{\C[\kappa_i]_{i\in I}}\C[\kappa_i,\kappa_i^{-1}]_{i\in I}\,.
\label{BeDr2}
\end{align}
\end{prop}
The standard Drinfeld generators $x^\pm_{i,r}$, $\phi^\pm_{i,\pm m}$,
$k_i^{\pm1}$ 
are related to those in the right hand side by 
\begin{align}
x^+_{i,r}=\tilde{x}^+_{i,r},\quad 
x^-_{i,r}=\kappa_i^{-1}\tilde{x}^-_{i,r},\quad 
\phi^\pm_{i,\pm m}=\kappa_i^{-1}\tilde{\phi}^\pm_{i,\pm m},\quad 
k_i=\kappa_i^{-1}\,.
\label{Dri-gen}
\end{align}
With this identification, we shall regard  $\Utg$ as a subalgebra of $U_q(\g)$. 

For $\g=\widehat{\mathfrak{sl}}_2$, 
the isomorphism \eqref{BeDr2} is given explicitly by 
\begin{align*}
&e_1\mapsto \tilde{x}_{1,0}^+,\quad f_1\mapsto \kappa_1^{-1}\tilde{x}_{1,0}^-,\quad 
e_0\mapsto \tilde{x}_{1,1}^-,\quad f_0\mapsto \tilde{x}_{1,-1}^+\kappa_1^{-1},
\\
&k_0\mapsto \kappa_1,\quad k_1\mapsto \kappa_1^{-1},
\quad [e_1,k_1e_0]\mapsto \frac{1}{q-q^{-1}} \kappa_1^{-1}\tilde{\phi}^+_{1,1}, 
\end{align*}
that is $k_1e_0 \mapsto x_{1,1}^-$  and  $f_0 k_1^{-1} \mapsto x_{1,-1}^+$. 

In the general case, 
the isomorphism \eqref{BeDr2} takes 
$e_i$ to $x^+_{i,0}$, 
$f_i$ to $x^-_{i,0}$, 
and $k_i$ to $\kappa_i^{-1}$
for $i\in I$. 
The image $y$ of $e_0$  
is described as follows (see e.g. \cite{Cha2}, p.393).  
In the Lie algebra $\gb$, choose simple root vectors 
$\dot{x}_i^+$ 
and suppose that the maximal root vector is written as a commutator 
$\lambda[\dot{x}^+_{i_1},[\dot{x}^+_{i_2},\cdots, [\dot{x}^+_{i_k},\dot{x}^+_{j_0}]\cdots]]$
with some $\lambda\in\C^{\times}$. 
Quite generally we say that an element $x$ of $U_q(\g)$ has weight $\beta\in Q$
if $k_ix=q^{(\beta,\alpha_i)}xk_i$ for all $i\in I$. 
For an element $A$ (resp., $B$) of weight $\alpha$ (resp., $\beta$), 
introduce the notation $[A,B]_q=AB-q^{(\alpha,\beta)}BA$. 
Then we have
\begin{align}
y = [\tilde{x}^-_{i_1,0},[\tilde{x}^-_{i_2,0},\cdots, [\tilde{x}^-_{i_k,0},
\tilde{x}^-_{j_0,1}]_q\cdots]_q]_q\,.
\label{e0}
\end{align}

In what follows we shall use the generating series 
\begin{align}
&\phi_i^\pm(z) = \sum_{m\geq 0}\phi_{i,\pm m}^\pm z^{\pm m},  
\quad
\tilde{\phi}_i^\pm(z) = \sum_{m\geq 0}\tilde{\phi}_{i,\pm m}^\pm z^{\pm
 m}.  
\label{phiz}
\end{align}

Let us define the following isomorphism.

\begin{prop} There is a unique isomorphism of $\CC$-algebras 
$$\sigma : U_q(\Glie)\rightarrow U_{q^{-1}}(\Glie)$$ 
satisfying
\begin{align*}
&\tilde{x}_{i,m}^+\mapsto q_i^2 \tilde{x}_{i,m}^-,\quad \tilde{x}_{i,m}^- \mapsto \tilde{x}_{i,m}^+,\quad 
\phi_{i,\pm r}^\pm \mapsto \phi_{i,\pm r}^\pm
\end{align*}
for $i\in I$, $m\in\ZZ$, $r\geq 0$.
\end{prop}

In particular we have $\sigma(k_i) = k_i$ and $\sigma(\tilde{\phi}_{i,\pm r}^{\pm}) = \tilde{\phi}_{i,\pm r}^\pm$
for $i\in I$ and $r\geq 0$. This implies that $\sigma$ restricts to an isomorphism of $\CC$-algebras $\tilde{U}_q(\Glie) \rightarrow \tilde{U}_{q^{-1}}(\Glie)$.

\begin{proof} As $\sigma(k_i) = k_i$, it suffices to check that the defining relations of $\tilde{U}_q(\Glie)$
are mapped to the defining relations of $\tilde{U}_{q^{-1}}(\Glie)$. 
This is immediate for all relations except 
\begin{align}
q_i^{-C_{i,j}}\tilde{x}_{i,r}^+\tilde{x}_{j,r'}^--\tilde{x}_{j,r'}^-\tilde{x}_{i,r}^+
= \delta_{i,j}\frac{\tilde{\phi}^+_{i,r+r'}-\tilde{\phi}^-_{i,r+r'}}{q_i-q_i^{-1}}.
\label{chk}
\end{align} 
The left hand side of \eqref{chk} is mapped to 
$$
q_i^{-C_{i,j}}q_i^2\tilde{x}_{i,r}^-\tilde{x}_{j,r'}^+
-q_i^2\tilde{x}_{j,r'}^+\tilde{x}_{i,r}^- = - q_i^{-C_{i,j} + 2}(- \tilde{x}_{i,r}^-\tilde{x}_{j,r'}^+
+ q_j^{C_{j,i}}\tilde{x}_{j,r'}^+\tilde{x}_{i,r}^-).
$$ 
If $i\neq j$, this is zero in $\tilde{U}_{q^{-1}}(\Glie)$. If $i = j$, this is
$$
- (q_i^2\tilde{x}_{i,r'}^+\tilde{x}_{i,r}^- - \tilde{x}_{i,r}^-\tilde{x}_{i,r'}^+).
$$
In this case, the right hand side of \eqref{chk} is mapped to 
$$
\frac{\tilde{\phi}^+_{i,r+r'}-\tilde{\phi}^-_{i,r+r'}}{q_i-q_i^{-1}} 
= -\frac{\tilde{\phi}^+_{i,r+r'}-\tilde{\phi}^-_{i,r+r'}}{q_i^{-1}-q_i},
$$
and so we get the correct relation in $\tilde{U}_{q^{-1}}(\Glie)$.
\end{proof}

For example, in the  case of $\mathfrak{sl}_2$, we get 
$\sigma(e_1) = q^2 \tilde{x}_{1,0}^-$, $\sigma(f_1) = k_1x_{1,0}^+$, 
$\sigma(e_0) = \sigma(\tilde{x}_{1,1}^-) = x_{1,1}^+$ and 
$\sigma(f_0) = \sigma(x_{1,-1}^+k_1) = q^2 \tilde{x}_{1,-1}^- k_1$. 

\subsection{Borel algebra}\label{boralg}

\begin{defi} The Borel algebra $U_q(\bo)$ is 
the subalgebra of $U_q(\g)$ generated by $e_i$ 
and $k_i^{\pm1}$ with $0\le i\le n$. 
\end{defi}
This is a Hopf subalgebra of $U_q(\g)$. 
It is well known 
(see e.g.\cite[Section 4.21]{ja}, \cite[Lemma 4]{h2})
that 
$U_q(\bo)$ is isomorphic to the algebra defined by the generators
$e_i$,  
$k^{\pm1}_i$ ($0\le i\le n$) and the relations
\begin{align*}
&k_ik_j=k_jk_i,\quad 
k_ie_j=q_i^{C_{i,j}}e_jk_i,
\\
&\sum_{r=0}^{s}(-1)^r 
(e_i)^{(s-r)}e_j(e_i)^{(r)}=0\quad
(s=1-C_{i,j}, \ i\neq j)
\,.
\end{align*}

For $\g = \widehat{\mathfrak{sl}}_2$, 
the Borel algebra $U_q(\bo)$ can be easily described
in terms of the Drinfeld generators : 
it is the subalgebra of $U_q(\widehat{\mathfrak{sl}}_2)$ 
generated by the $x_{1,m}^+$, $x_{1,r}^-$, $k_1^{\pm 1}$, 
$\phi_{1,m}^+$ where $m \geq 0$ and $r > 0$.
Such a simple description does not hold in the general case. 
Nevertheless it is known \cite{bec} that the algebra $U_q(\mathfrak{b})$ 
contains 
the Drinfeld generators
$x_{i,m}^+$, $x_{i,r}^-$, $k_i^{\pm 1}$, $\phi_{i,r}^+$ 
where $i\in I$, $m \geq 0$ and $r > 0$. 

Let $U_q(\g)^\pm$ 
(resp. $U_q(\g)^0$) be the subalgebra of 
$U_q(\g)$ generated by the 
$x_{i,r}^\pm$ where $i\in I, r\in\ZZ$ 
(resp. by the $k_i^{\pm 1}$, $\phi_{i,\pm r}^\pm$ where $i\in I,r> 0$). 

Then by \cite{bec}, 
\cite[Proposition 1.3]{bcp} we have a triangular 
decomposition (isomorphism of vector spaces) :
$$
U_q(\g)\simeq 
U_q(\g)^-\otimes U_q(\g)^0 \otimes U_q(\g)^+.
$$
Let further $U_q(\bo)^\pm = U_q(\g)^\pm\cap U_q(\bo)$ and
$U_q(\bo)^0 = U_q(\g)^0\cap U_q(\bo)$. 
Then we have 
\begin{align*}
U_q(\bo)^+ = \langle x_{i,m}^+\rangle_{i\in I, m\geq 0},\quad
U_q(\bo)^0 = 
\langle\phi_{i,r}^+,k_i^{\pm1}\rangle_{i\in I, r>0}.
\end{align*} 
In general, $U_q(\bo)^-$ does not have such a nice description 
in terms of Drinfeld generators, except when $\g = \widehat{\mathfrak{sl}}_2$
for which 
$U_q(\bo)^- = \langle x_{1,m}^-\rangle_{m\geq 1}$.

By using the PBW basis of \cite{bec}, we have a triangular decomposition
\begin{equation}\label{triandecomp}
U_q(\bo)\simeq U_q(\bo)^-\otimes 
U_q(\bo)^0 \otimes U_q(\bo)^+.
\end{equation}

\subsection{From $\Utg$ to $U_q(\mathfrak{b})$}\label{subsec:Utg}

A representation $V$ of the asymptotic algebra $\Utg$ is said to be 
$Q$-graded if there is a decomposition into a direct sum of linear subspaces 
$V = \bigoplus_{\alpha\in Q} V^{(\alpha)}$ such that 
\begin{align*}
\tilde{x}_{i,r}^\pm V^{(\alpha)}\subset V^{(\alpha \pm \alpha_i)},\quad 
\tilde{\phi}_{i,\pm m}^\pm V^{(\alpha)}\subset V^{(\alpha)}, 
\quad \kappa_i V^{(\alpha)}\subset V^{(\alpha)}
\end{align*}
hold for any $\alpha\in Q$, $i\in I$, $r\in\ZZ, m\geq 0$.

\begin{prop}\label{asym} 
For a $Q$-graded $\Utg$-module  $V$, 
there is a unique structure of $U_q(\mathfrak{b})$-module on $V$ 
such that 
\begin{align}
e_i\, v = \tilde{x}_{i,0}^+v,\quad e_0\, v = y\, v,
\quad
k_iv = q_i^{\alpha(\alpha_i^{\vee})}v\,
\quad (i\in I,\ v\in V^{(\alpha)})\,,
\label{vsigma1}
\end{align}
where the element $y\in \Utg$ is given in \eqref{e0}.

Similarly, for a $Q$-graded $\tilde{U}_{q^{-1}}(\Glie)$-module  $V$, 
 there is a unique structure of $U_q(\mathfrak{b})$-module on $V$ such that
\begin{align}
e_i\,v = \sigma(\tilde{x}_{i,0}^+)v,\quad e_0\,v = \sigma(y) v,
\quad
k_iv = q_i^{-\alpha(\alpha_i^{\vee})}v\, 
\quad (i\in I,\ v\in V^{(\alpha)})\,.
\label{vsigma2}
\end{align}
\end{prop}

\begin{proof}
Let us check that the relations of $U_q(\mathfrak{b})$ are satisfied.
Since each $V^{(\alpha)}$ is a joint eigenspace of  $k_i$, it is clear that $k_ik_j=k_jk_i$.  
From $\tilde{x}_{i,r}^\pm V^{(\alpha)}\subset V^{(\alpha \pm \alpha_i)}$ ,
we get the relations  
$k_ie_j=q_i^{C_{i,j}}e_jk_i$. 
Finally the elements $y,\tilde{x}^+_{1,0},\cdots,\tilde{x}^+_{n,0}$ satisfy 
Serre relations in the algebra $\Utg$ (resp. $\tilde{U}_{q^{-1}}(\Glie)$).
\end{proof}

In general, the action of $U_q(\mathfrak{b})$ can not be extended to 
an action of 
the full quantum affine algebra $U_q(\Glie)$, 
because the $\kappa_i$'s are allowed to have non-trivial kernels on $V$. 

As $\sigma(U_{q^{-1}}(\bo))\neq U_q(\bo)$, it is not sufficient to
have a $U_{q^{-1}}(\bo)$-module structure on $V$ to get a $U_q(\bo)$-module structure
via $\sigma$. That is why we will construct $\tilde{U}_{q^{-1}}(\Glie)$-modules.

\begin{rem}\label{grading-shift} 
For a $Q$-graded $U_q(\bo)$-module 
$V=\oplus_{\alpha\in Q}V^{(\alpha)}$, one can freely 
shift the action of the $k_i$'s. 
Namely, for any $\beta_i\in \C^{\times}$ ($i\in I$), 
 a new $U_q(\bo)$-module structure on $V$ is obtained
by setting 
\begin{align*}
k_i\ v=\beta_i\, q^{\alpha(\alpha_i)}v\qquad (v\in V^{(\alpha)})\,
\end{align*}
and retaining the same action of $e_i$'s. 
\end{rem}

\subsection{Coproduct}

The algebra $U_q(\Glie)$ has a natural $Q$-grading defined by
\begin{align*}
&\deg\bigl(x^\pm_{i,m}\bigr) = \pm\alpha_i,
\quad
\deg\bigl(h_{i,r}\bigr)=
\deg\bigl(k_i^\pm\bigr) = 0\,.
\end{align*}
Let $\U_q^+(\Glie)$ (resp. ${\U}_q^-(\Glie)$) be the subalgebra of $U_q(\Glie)$
consisting of elements of positive (resp. negative) $Q$-degree. 
These subalgebras should not be confused with the subalgebras $U_q(\Glie)^\pm$ 
previously defined in terms of Drinfeld generators. 
Let 
$$
X^+ = \sum_{j\in I, m\in\ZZ}\CC x_{j,m}^+\subset {\U}_q^+(\Glie).
$$

\begin{thm}\cite{da}\label{apco} Let $i\in I$. For $m\in\ZZ$, we have
\begin{equation}\label{x}\Delta\left(x_{i,m}^+\right) \in x_{i,m}^+\otimes 1 + U_q(\Glie)
\otimes \left(U_q(\Glie) X^+\right).\end{equation}
For $r > 0$, we have
\begin{equation}\label{h}
\Delta\left(\phi_{i,\pm r}^\pm \right) \in 
\sum_{0\leq l\leq r} \phi_{i,\pm l}^\pm \otimes \phi_{i,\pm (r - l)}^\pm
+ \U_q^-(\Glie) \otimes \U_q^+(\Glie),
\end{equation}
\begin{equation}\label{mplus}
\Delta\left(x_{i,r}^-\right) \in x_{i,r}^-\otimes k_i + 1 \otimes x_{i,r}^- + \sum_{1\le j < r}
x_{i,r - j}^-\otimes \phi_{i,j}^+ + U_q(\Glie)\otimes \left(U_q(\Glie) X^+\right).
\end{equation}
For $r\le 0$, we have
\begin{equation}\label{mmoins}
\Delta\left(x_{i,r}^-\right) \in x_{i,r}^-\otimes k_i^{-1} + 1 \otimes x_{i,r}^- + \sum_{1\le j \le -r}
x_{i,r + j}^-\otimes \phi_{i,-j}^- + U_q(\Glie)\otimes \left(U_q(\Glie) X^+\right).
\end{equation}
\end{thm}

\section{Category $\mathcal{O}$ for $U_q(\mathfrak{b})$}\label{catO}

\subsection{Highest $\ell$-weight modules}\label{lhwm}

Denote by $\tb$ the subalgebra of $U_q(\bo)$ generated by $\{k_j^{\pm1}\}_{j\in I}$. 
Set $\tb^*=\bigl(\C^\times\bigr)^I$, and endow it with a group structure
by pointwise multiplication. 
We define a group morphism $\overline{\phantom{u}}:Q \longrightarrow \tb^*$ by setting  
$\ga_i(j)=q_i^{C_{i,j}}$ for a simple root $\alpha_i$. 
We shall use the standard partial ordering  on $\tb^*$:
\begin{align}
\omega\leq \omega' \quad \text{if $\omega \omega'^{-1}$ 
is a product of $\{\ga_i^{-1}\}_{i\in I}$}.
\label{partial}
\end{align}

For a $U_q(\mathfrak{b})$-module $V$ and $\omega\in \tb^*$, we set
\begin{align}
V_{\omega}=\{v\in V \mid  k_i\, v = \omega(i) v\ (\forall i\in I)\}\,,
\label{wtsp}
\end{align}
and call it the weight space of weight $\omega$. 
For any $i\in I$, $r\in\ZZ$ we have $\phi_{i,r}^\pm (V_\omega)\subset V_\omega$
and $x_{i,r}^\pm (V_{\omega}) \subset V_{\omega \ga_i^{\pm 1}}$.
We say that $V$ is $\tb$-diagonalizable 
if $V=\underset{\omega\in \tb^*}{\bigoplus}V_{\omega}$.

\begin{defi} A series $\Psib=(\Psi_{i, m})_{i\in I, m\geq 0}$ 
of complex numbers such that 
$\Psi_{i,0}\neq 0$ for all $i\in I$ 
is called an $\ell$-weight. 
\end{defi}

We denote by $\tb^*_\ell$ the set of $\ell$-weights. 
Identifying $(\Psi_{i, m})_{m\geq 0}$ with its generating series we shall
write
\begin{align*}
\Psib = (\Psi_i(z))_{i\in I},
\quad
\Psi_i(z) = \underset{m\geq 0}{\sum} \Psi_{i,m} z^m.
\end{align*}
Since each $\Psi_i(z)$ is an invertible formal power series,
$\tb^*_\ell$ has a natural group structure. 
We have a surjective morphism of groups
$\varpi : \tb^*_\ell\rightarrow \tb^*$ given by 
$\varpi(\Psib)(i)=\Psi_{i,0}$.
For a $U_q(\mathfrak{b})$-module $V$ and $\Psib\in\tb_\ell^*$, 
the linear subspace
\begin{align}
V_{\Psibs} =
\{v\in V\mid
\exists p\geq 0, \forall i\in I, 
\forall m\geq 0,  
(\phi_{i,m}^+ - \Psi_{i,m})^pv = 0\}
\label{l-wtsp} 
\end{align}
is called the $\ell$-weight space of $V$ of $\ell$-weight $\Psib$. 

\begin{defi} A $U_q(\mathfrak{b})$-module $V$ is said to be 
of highest $\ell$-weight 
$\Psib\in \tb^*_\ell$ if there is $v\in V$ such that 
$V =U_q(\mathfrak{b})v$ 
and the following hold:
\begin{align*}
e_i\, v=0\quad (i\in I)\,,
\qquad 
\phi_{i,m}^+v=\Psi_{i, m}v\quad (i\in I,\ m\ge 0)\,.
\end{align*}
\end{defi}

The $\ell$-weight $\Psib\in \tb^*_\ell$ is uniquely determined by $V$. 
It is called the highest $\ell$-weight of $V$. 
The vector $v$ is said to be a highest $\ell$-weight vector of $V$.
See \cite{nams, h2} for an analogous definition in the 
context of representations of quantum affinizations. 

\begin{lem} Let $V$ be a highest $\ell$-weight $U_q(\mathfrak{b})$-module 
with highest $\ell$-weight vector $v$. Then $x_{i,m}^+v=0$ 
holds for all $i\in I, m\geq 0$.

Consequently $V = U_q(\mathfrak{b})^-v$.
Moreover $V$ is $\tb$-diagonalizable 
and $V=\underset{\lambda\leq \varpi(\Psibs)}{\bigoplus}V_{\lambda}$.
\end{lem}

\begin{proof} 
The first statement can be verified by induction using 
the formula
\begin{align*}
[\tilde{\phi}^+_{i,1},x^+_{i,m}]=(q_i^2-q_i^{-2})x^+_{i,m+1}\,.
\end{align*}
The rest of the assertions are clear from 
(\ref{triandecomp}).
\end{proof}

\medskip

\noindent{\it Example.}\quad 
For any $\Psib\in \tb^*_\ell$, define 
the Verma module $M(\Psib)$ to be the quotient of 
$U_q(\mathfrak{b})$ 
by the left ideal generated by $e_i$  
($i\in I$) and $\phi_{i,m}^+-\Psi_{i,m}$ 
($i\in I, m\geq 0$). 
From (\ref{triandecomp}), $M(\Psib)$ is a 
free $U_q(\mathfrak{b})^-$-module 
of rank $1$. In particular it is non trivial and it is a highest $\ell$-weight module of highest $\ell$-weight $\Psib$. 

\subsection{Simple highest $\ell$-weight modules}

Let $\Psib\in\tb^*_\ell$. By a standard argument, the Verma module
$M(\Psib)$ has a unique proper submodule, and so we get the following.

\begin{prop}\label{simple} 
For any $\Psib\in \tb^*_\ell$, there exists a simple 
highest $\ell$-weight module $L(\Psib)$ of highest $\ell$-weight $\Psib$. This module is unique up to isomorphism.  
\end{prop}

Let us give two fundamental examples.

\medskip

\noindent{\it Example.}\quad 
Let $\omega\in\tb^*$. We define $\Psib_\omega\in\tb^*_\ell$
by $(\Psib_\omega)_{i,0} = \omega(i)$, $(\Psib_\omega)_{i,m} = 0$ for $i\in I$ and $m\geq 1$.
Then $L(\Psib_\omega) = \CC v$ is $1$-dimensional.
\medskip

\noindent{\it Example.}\quad 
For $i\in I$, let $P_i(z)\in\mathbb{C}[z]$ be a polynomial with constant 
term $1$. Set
\begin{align*}
\Psib = (\Psi_i(z))_{i\in I},
\quad
\Psi_i(z) = 
q_i^{\text{deg}(P_i)}\frac{P_i(zq_i^{-1})}{P_i(zq_i)}.
\end{align*}
Then $L(\Psib)$ is finite-dimensional. 
Moreover the action of $U_q(\mathfrak{b})$
can be uniquely extended to an action of the full quantum affine algebra 
$U_q(\Glie)$, and  
all (type $1$) irreducible finite-dimensional $U_q(\Glie)$-modules
are of this form.
This follows from the classification of simple finite dimensional 
modules 
of quantum affine algebras \cite{Cha2} by Drinfeld polynomials 
along with the following result.

\begin{prop}\label{fd} Let $V$ be a simple finite dimensional 
 $U_q(\g)$-module. 
 Then $V$ is simple as a $U_q(\mathfrak{b})$-module.
\end{prop}
This result was proved in \cite{bt} for $\g=\widehat{\mathfrak{sl}}_2$, 
and in \cite{bo}, \cite[Proposition 2.7]{cg} in the general case. 
For completeness, let us give a short elementary proof of this statement, independent from 
the proof of \cite{bo}.

\begin{proof}
Let $\pi: U_q(\Glie)\rightarrow \mathrm{End}(V)$ be the representation morphism.
Let $i\in I$.
We prove for any $r\in \ZZ$ that $\pi(x_{i,r}^\pm)$ is in the space
linearly spanned by $(\pi(x_{i,m}^\pm))_{m > 0}$. 
Since $\text{End}(V)$ is finite dimensional, there is a non-trivial linear relation
$\sum_{r=a}^b\lambda_r \pi(x_{i,r}^\pm) =0$ with $a\geq 1$.
On the other hand, for any $r\in\ZZ$ we have
\begin{align*}
\phi_{i,-1}^- x_{i,r}^{\pm} - q_i^{\mp 2} 
x_{i,r}^{\pm}\phi_{i,-1}^-
= (q_i^{\mp 4} - 1 )x_{i,r-1}^{\pm} \phi_{i,0}^-\,.
\end{align*}
It follows that
for any $R\ge 0$ we have $\sum_{r=a}^{b}\lambda_r \pi(x_{i,r-R}^\pm) =0$.
Now it is clear that $V$ is cyclic as a $U_q(\mathfrak{b})$-module generated by a highest $\ell$-weight vector, 
and that all primitive vectors are of highest $\ell$-weight.
\end{proof}
Note that in the proof, we proved a weak version of the quasi-polynomiality property 
(see \cite[Proposition 6.2]{beckac}, \cite[Proposition 3.8]{lms}). 

\begin{rem}\label{mult} 
From the relation (\ref{h}), the submodule of $L(\Psib)\otimes L(\Psib')$ generated by the tensor product
of the highest $\ell$-weight vectors is of highest $\ell$-weight $\Psib\Psib'$. In particular, $L(\Psib\Psib')$
is a subquotient of $L(\Psib)\otimes L(\Psib')$.
\end{rem}

\begin{defi}
For $i\in I$ and $a\in\CC^\times$, let 
\begin{align}
L_{i,a}^\pm = L(\Psib)
\quad \text{where}\quad 
\Psi_j(z) = \begin{cases}
(1 - za)^{\pm 1} & (j=i)\,,\\
1 & (j\neq i)\,.\\
\end{cases} 
\label{fund-rep}
\end{align}
The representations $L_{i,a}^\pm$ ($i\in I, a\in\CC^\times$)  
are called fundamental representations. 
\end{defi}

For $a\in\CC^\times$, we have an algebra automorphism 
$\tau_a:U_q(\g)\longrightarrow U_q(\g)$
such that
\begin{align}
\tau_a(x_{i,m}^\pm) = a^m x_{i,m}^\pm\,, 
\quad
\tau_a(\phi_i^\pm(z)) =  \phi_{i}^\pm(az)\,.
\label{tau-a}
\end{align}
The subalgebra $U_q(\mathfrak{b})$ is stable by $\tau_a$.
Denote its restriction to $U_q(\mathfrak{b})$ by the same letter.
Then the pullbacks of the  $U_q(\mathfrak{b})$-modules $L_{i,b}^\pm$ by $\tau_a$ 
is $L_{i,ab}^\pm$. 

We remark that for $a = 0$, we get an algebra morphism
$\tau_0:U_q(\mathfrak{b})\longrightarrow U_q(\mathfrak{b})$
which is not invertible. Its image is the subalgebra 
$U_q(\dot{\mathfrak{b}})$ of $U_q(\mathfrak{b})$
generated by the $x_{i,0}^-$, $k_i^{\pm 1}$ ($i\in I$), 
that is the
Borel algebra of the quantum algebra of finite type 
associated to $U_q(\gb)$. The pullback of the $1$-dimensional simple
representations of $U_q(\dot{\mathfrak{b}})$ are the 
representations $L(\Psib_\omega)$.

\subsection{Category $\mathcal{O}$}\label{precato}

For $\lambda\in \tb^*$, we set $D(\lambda )=
\{\omega\in \tb^* \mid \omega\leq\lambda\}$.

\begin{defi} A $U_q(\mathfrak{b})$-module $V$ 
is said to be in category $\mathcal{O}$ if:

i) $V$ is $\tb$-diagonalizable,

ii) for all $\omega\in \tb^*$ we have 
$\dim (V_{\omega})<\infty$,

iii) there exist a finite number of elements 
$\lambda_1,\cdots,\lambda_s\in \tb^*$ 
such that the weights of $V$ are in 
$\underset{j=1,\cdots, s}{\bigcup}D(\lambda_j)$.
\end{defi}

The category $\mathcal{O}$ is a tensor category. 
In general a simple highest $\ell$-weight module is not necessarily 
in category $\mathcal{O}$. 

\begin{lem}\label{if} Let $V$ be a $U_q(\mathfrak{b})$-module, $\omega\in \tb^*$
and $i\in I$. We suppose that $V_\omega$, 
$V_{\omega \overline{\alpha}_i}$ and 
$V_{\omega \overline{\alpha}_i^{-1}}$
are finite-dimensional. Then for $\Psib\in \tb^*_\ell$ such that $\varpi(\Psib) = \omega$
and $V_{\Psibs}\neq 0$, $\Psi_i(z)$ is rational.
\end{lem}

\begin{proof}
The proof is a generalization of the proof of \cite[Lemma 14]{h2}.

There is a non zero $v\in V_{\Psibs}$ such that $\phi_{i,m}^+v = \Psi_{i,m}v$
for any $i\in I$, $m\geq 0$. 
Choose a basis $(u_1,\cdots,u_R)$ 
of $V_{\omega \overline{\alpha}_i^{-1}}$, a basis 
$(v_1,\cdots,v_P)$ of $V_{\omega \overline{\alpha}_i}$, 
and a linear map $\pi : V_\omega\rightarrow \CC v$
such that $\pi(v)=v$. 

For $m\geq 0$ and $p\geq 1$, we can find 
$\lambda_{p,j},\lambda_{p,j}',\mu_{m,j}, \mu_{m,j}'\in\mathbb{C}$ so that
\begin{align*}
&x_{i,p}^-v=\sum_{k=1}^R\lambda_{p,k}u_k,\quad
(q_i - q_i^{-1})\pi\bigl(x_{i,p}^-v_j\bigr) 
=\lambda_{p,j}' v\quad (1\le j\le P)\,,
\\
&x_{i,m}^+v=\sum_{j=1}^P\mu_{m,j}v_j,\quad
(q_i - q_i^{-1})\pi \bigl(x_{i,m}^+u_k\bigr)
=\mu_{m,k}' v\quad(1\le k\le R).
\end{align*}
Using 
$(q_i - q_i^{-1})[x_{i,m}^+,x_{i,p}^-]v=\Psi_{i,m+p}v$
and applying $\pi$ on both sides 
we obtain 
\begin{align*}
\Psi_{i,m+p} = 
\sum_{k=1}^R\lambda_{p,k}\mu'_{m,k}-\sum_{j=1}^P\mu_{m,j}\lambda_{p,j}'\,.
\end{align*}

We set $\lambda_k(z)=\underset{p\geq 1}{\sum}\lambda_{p,k} z^p$, 
$\lambda_j'(z)=\underset{p\geq 1}{\sum}\lambda_{p,j}' z^p$ and 
\begin{align*}
\Psi_i^{> m}(z) = z^{-m}\bigl(\Psi_i(z) - \sum_{p=0}^{m}\Psi_{i,p} z^p\bigr)
\quad (m\ge 0)\,.
\end{align*}
Then for $m \ge 0$ we have 
$$
\Psi_i^{ > m}(z)= 
\sum_{k=1}^R\lambda_k(z)\mu_{m,k}' -  
\sum_{j=1}^P\lambda'_j(z)\mu_{m,j}.
$$ 
Hence $(\Psi_i^{ > m}(z))_{m \ge 0}$ is not linearly independent,  and 
we have a relation of the form 
\begin{align*}
\sum_{m=0}^Na_m z^{N-m}
\bigl(\Psi_i(z) - \sum_{p=0}^{m}\Psi_{i,p} z^p\bigr)=0.
\end{align*}
This shows that $\Psi_i(z)$ is rational.
\end{proof}

As a direct consequence, we have the following.

\begin{prop}\label{rat} Let $V$ be in category $\mathcal{O}$. 
Then for $\Psib\in \tb^*_\ell$ such that
$V_{\Psibs}\neq 0$, $\Psi_i(z)$ is rational for any $i\in I$.
\end{prop}

The following is one of the main results of this paper. It is a complete classification
of simple objects in category $\mathcal{O}$.

\begin{thm}\label{class} For $\Psib\in\tb^*_\ell$, 
the simple module $L(\Psib)$ is in category 
$\mathcal{O}$ if and only if $\Psi_i(z)$ 
is rational for any $i\in I$.
\end{thm}
\begin{proof}
The ``only if" part follows directly from Proposition \ref{rat}.

From Remark \ref{mult}, 
an arbitrary simple representation in category $\mathcal{O}$ is a subquotient of a tensor product
of fundamental representations. 
So, to prove the ``if'' part, it suffices to show that all fundamental representations
are in category $\mathcal{O}$.
Furthermore, with the aid of the twist automorphism \eqref{tau-a}, 
the proof is reduced to the case of $L_{i,1}^\pm$. 

In the next sections we shall show that $L_{i,1}^\pm$ are indeed in category $\mathcal{O}$
(see Corollary \ref{Lmin} and Corollary \ref{Lpl}).
\end{proof}

\subsection{$q$-characters in category $\mathcal{O}$}\label{qcharo}

\newcommand{\mfr}{\mathfrak{r}}

Let 
$\mfr$
be the subgroup of $\tb^*_\ell$
consisting of $\Psib$ such that 
$\Psi_i(z)$ is rational for any $i\in I$.
(The letter $\mfr$  
stands for `rational'.)

Let 
$\mathcal{E}_\ell\subset \Z^{\mfr}$
be the ring of maps
$c : \mfr\rightarrow \ZZ$  
satisfying $c(\Psib) = 0$ for all 
$\Psib$ such that $\varpi(\Psib)$ is outside a finite union
of sets of the form $D(\mu)$ and such that for 
each $\omega\in \tb^*$, there are finitely many $\Psib$
such that $\varpi(\Psib) = \omega$ and $c(\Psib)\neq 0$.
Similarly, let $\mathcal{E}\subset \Z^{\tb^*}$ 
be the ring of maps $c : \tb^* \rightarrow \ZZ$ 
satisfying 
$c(\omega) = 0$ for all $\omega$ outside 
a finite union of sets of the form $D(\mu)$. 
The map $\varpi$ is naturally extended to a surjective ring morphism 
$\varpi : \mathcal{E}_\ell\rightarrow \mathcal{E}$. 

For $\Psib\in\mfr$  
(resp. $\omega\in\tb^*$), we define 
$[\Psib] = \delta_{\Psibs,.}\in\mathcal{E}_\ell$ 
(resp. $[\omega] = \delta_{\omega,.}\in\mathcal{E}$).

Let $V$ be a $U_q(\mathfrak{b})$-module in category $\mathcal{O}$. 
We define the $q$-character of $V$ to be 
the element of $\mathcal{E}_\ell$ 
\begin{align}
\chi_q(V) = 
\sum_{\Psibs\in\mfr}  
\mathrm{dim}(V_{\Psibs}) [\Psib]\,.
\label{qch}
\end{align}
Similarly we define the ordinary character of $V$ to be
an element of $\mathcal{E}$
\begin{align}
\chi(V) = \varpi(\chi_q(V)) =  \sum_{\omega\in\tb^*} 
\text{dim}(V_\omega) [\omega]\,.
\label{ch}
\end{align}
For $V$ in category $\mathcal{O}$ which has a unique $\ell$-weight $\Psib$
whose weight is maximal (for example a highest $\ell$-weight module),
we also consider its normalized $q$-character $\tilde{\chi}_q(V)$
and normalized character $\tilde{\chi}(V)$ by
\begin{align*}
\tilde{\chi}_q(V) = [\Psib^{-1}]\cdot\chi_q(V)\,,
\quad 
\tilde{\chi}(V) = \varpi(\tilde{\chi}_q(V))\,.
\end{align*}

Let $\text{Rep}(U_q(\mathfrak{b}))$ be the Grothendieck ring of the
category $\mathcal{O}$.  
We define the $q$-character morphism as the group morphism
$$
\chi_q : \text{Rep}(U_q(\mathfrak{b}))\rightarrow \mathcal{E}_\ell
$$
which sends a class of a representation $V$ to $\chi_q(V)$.
The map is well-defined as $\chi_q$ is clearly 
compatible with exact sequences.

By using Theorem \ref{apco}, we can prove the following as in 
\cite[Lemma 3]{Fre} and \cite[Theorem 3]{Fre}.

\begin{prop} The $q$-character morphism is an injective ring morphism.
\end{prop}

\begin{rem} 
As $\mathcal{E}_\ell$ is clearly a commutative ring, this implies that 
$\text{Rep}(U_q(\mathfrak{b}))$ is commutative. The argument 
is the same as for the commutativity 
of the Grothendieck ring $\text{Rep}(U_q(\Glie))$ 
of finite-dimensional modules 
of $U_q(\Glie)$ given in \cite{Fre}. Note that $\text{Rep}(U_q(\Glie))$
is naturally a subring of $\text{Rep}(U_q(\mathfrak{b}))$. 
The category $\mathcal{O}$ is not braided
(it is easy to construct a counter-example by using the category of finite 
dimensional representations of $U_q(\Glie)$ which is known 
 to be not braided). 
Moreover, in contrast to the case of quantum affine algebras, 
no meromorphic $R$-matrix (in the sense of \cite{ks})
is known for $U_q(\mathfrak{b})$. 
That is why the commutativity
is a bit more surprising in this context.
\end{rem}

\subsection{Finite dimensional representations of $U_q(\Glie)$}

In this subsection we quote some results 
for finite-dimensional $U_q(\Glie)$-modules
which will be used in subsequent sections. 
For more details, the reader may refer to the book \cite{Cha2} 
and to the recent review papers \cite{ch, le}.
We consider only representations  
of type $1$, namely such that the eigenvalues of the
$k_i$ ($i\in I$) are in $q^{\ZZ}$.

We have reminded above the parametrization of simple irreducible
$U_q(\Glie)$-modules by $n$-tuples $(P_i(z))_{i\in I}$ of polynomials
of constant term $1$. 

Following \cite{Fre}, 
consider the ring of Laurent polynomials
$\Yim = \ZZ[Y_{i,a}^{\pm 1}]_{i\in I,a\in\CC^*}$ 
in the indeterminates $\{Y_{i,a}\}_{i\in I, a\in \C^*}$.
Let  
$\mathcal{M}$ be the group of monomials of $\Yim$. 
For a monomial 
$m = \prod_{i\in I, a\in\CC^*}Y_{i,a}^{u_{i,a}}$, 
we consider its `evaluation on $\phi^+(z)$'. 
By definition it is an element 
$m(\phi(z))\in\mfr$  
given by
\begin{align*}
&m\bigl(\phi(z))=
\prod_{i\in I, a\in\CC^*}
\left(Y_{i,a}(\phi(z))\right)^{u_{i,a}},
\end{align*}
where
\begin{align*}
&\Bigl(Y_{i,a}\bigl(\phi(z)\bigr)\Bigr)_j
=\begin{cases}
\displaystyle{q_i\frac{1-a q_i^{-1}z}{1-aq_iz}}& (j=i),\\
1 & (j\neq i).\\
\end{cases}
\end{align*}
This defines an injective group morphism 
$\mathcal{M}\rightarrow \mfr$. 
We identify a monomial $m\in\mathcal{M}$ with its image in 
$\mfr$. 
Note that $\varpi(Y_{i,a}) = \overline{\omega_i}$.

It is proved in \cite{Fre} that a finite-dimensional
$U_q(\Glie)$-module $V$ satisfies 
$V = \bigoplus_{m\in\mathcal{M}} V_{m\left(\phi(z)\right)}$.
In particular, $\chi_q(V)$ can be viewed  
as an element of $\Yim$.

Note that for $V$ a finite dimensional $U_q(\Glie)$-module 
and $\Psib\in\mfr$,  
the $\ell$-weight spaces \eqref{l-wtsp} 
can be characterized alternatively as \cite{Fre} 
$$
V_{\Psibs} = \{v\in V\mid \exists p\geq 0, 
\forall i\in I, \forall m\geq 0, (\phi_{i,-m}^- - \Psi_{i,-m}^-)^pv = 0\}\,.
$$ 
Here 
$\Psi_i^-(z^{-1}) = \sum_{m\geq 0}\Psi_{i,-m}^-z^{-m}$ is 
the expansion of $\Psi_i(z)\in\CC(z)$ in $z^{-1}$.

A monomial $M\in\mathcal{M}$ is said to be dominant if  
$M\in\ZZ[Y_{i,a}]_{i\in I, a\in\CC^*}$.
For $L(\Psib)$ a finite-dimensional simple $U_q(\Glie)$ module, 
$\Psib = M\bigl(\phi(z)\bigr)$ holds for some dominant monomial $M\in\mathcal{M}$.  
The representation will be denoted by $L(M)$.

For example, for $i\in I$, $a\in\CC^*$ and $k\geq 0$, let 
\begin{align}
M_{k,a}^{(i)} = Y_{i,a}Y_{i,aq_i^2}\cdots Y_{i,aq_i^{2(k-1)}}\,.
\label{KRmod}
\end{align}
Then $W_{k,a}^{(i)} = L(M_{k,a}^{(i)})$ is called a Kirillov-Reshetikhin module (KR module).
An explicit formula for $\chi(W_{k,a}^{(i)})$ is known \cite{n, hcr}.

For $i\in I, a\in\CC^*$, 
define $A_{i,a}\in\mathcal{M}$ to be 
$$
Y_{i,aq_i^{-1}}Y_{i,aq_i}
\Bigl(\prod_{\{j\in I|C_{j,i} = -1\}}Y_{j,a}
\prod_{\{j\in I|C_{j,i} = -2\}}Y_{j,aq^{-1}}Y_{j,aq}
\prod_{\{j\in I|C_{j,i} =
-3\}}Y_{j,aq^{-2}}Y_{j,a}Y_{j,aq^2}\Bigr)^{-1}\,.
$$
Note that $\varpi(A_{i,a}) = \overline{\alpha_i}$. 

\begin{thm}\label{trunc}\cite{Fre, Fre2} Le $V$ a simple finite-dimensional $U_q(\Glie)$-module. Then
$$\tilde{\chi}_q(V)\in \ZZ[A_{i,a}^{-1}]_{i\in I, a\in\CC^*}.$$
\end{thm}

Let $M=\prod_{i\in I, \atop r\in \Z}Y_{i,q^r}^{u_{i,r}}$ 
be a dominant monomial, and let $l\in \Z$. 
We set
\begin{align*}
M^{\ge l}=\prod_{i\in I, r\ge l}Y_{i,q^r}^{u_{i,r}}, 
\quad 
M^{< l}=\prod_{i\in I, r< l}Y_{i,q^r}^{u_{i,r}}, 
\end{align*}
so that $M=M^{\ge l}M^{< l}$.  
It is well-known that the results in \cite{c, kas, vv}  
imply the existence
of a surjective morphism of $U_q(\Glie)$-modules 
(see references in \cite[Corollary 5.5]{h3}) :
\begin{align}
L(M^{\ge l}) \otimes L(M^{<l})\longrightarrow L(M)\,.
\label{LLL}
\end{align}
Consider the linear subspace 
\begin{align*}
L(M)^{\ge l}=\bigoplus_{m\in M\Z[A^{-1}_{i,q^{r+d_i}}
]_{i\in I,r\ge l}}(L(M))_m\,.
\end{align*}
The following result will play a key role in subsequent sections. 
\begin{thm}\cite{h3}\label{H3}
(i) The map \eqref{LLL} restricts to an isomorphism 
of vector spaces
\begin{align}
L(M^{\ge l})\otimes v^{<l}\longrightarrow L(M)^{\ge l},
\label{LL2}
\end{align}
where $v^{<l}$ is a highest weight vector of $L(M^{<l})$.

(ii) Let $F:L(M^{\ge l})\to L(M)^{\ge l}$ be the composition of the map \eqref{LL2} with
the natural map $L(M^{\ge l})\to L(M^{\ge l})\otimes v^{<l}$. 
Then for any $i,j\in I$ and $r\in \Z$
we have
\begin{align}
&x^+_{i,r}\, F=F\, x^+_{i,r},
\label{Fx}
\\
&\phi_j^{\pm}(z)\, F=F\, M^{<l}\bigl(\phi^\pm_j(z)\bigr)\times \phi^\pm_j(z)\,.
\label{Fphi}
\end{align}
\end{thm}
Statement (i) appears as Proposition 5.6 in \cite{h3}, and statement 
(ii) follows from Remark 5.7, {\it loc. cit.}. A particular
case of this result had been proved in \cite{hl}.

\subsection{The dual category $\mathcal{O}^*$}\label{dualcat}

For $V$ a $\tb$-diagonalizable $U_q(\bo)$-module, 
we define a structure of $U_q(\bo)$-module on its
graded dual $V^* = \oplus_{\beta\in \tb^*} V_\beta^*$ by
\begin{align*}
(x\,u)(v)=u\bigl(S^{-1}(x)v\bigr)\quad
(u\in  V^*, \ v\in V,\ x \in U_q(\mathfrak{b})).
\end{align*}
The reason why we use $S^{-1}$ and not the antipode in the definition of $V^*$ 
is discussed in Remark \ref{smoins} below. 

\begin{defi} Let $\mathcal{O}^*$ be the category of $\tb$-diagonalizable $U_q(\bo)$-modules $V$ 
such that $V^*$ is in category $\mathcal{O}$.
\end{defi}

\begin{lem}\label{autredef} A $\tb$-diagonalizable $U_q(\bo)$-module $V$ is in category $\mathcal{O}$
if and only if $V^*$ is in category $\mathcal{O}^*$.
\end{lem}

\begin{proof} As $S^2(k_i) = k_i$ for any $i\in I$, the weight spaces of $(V^*)^*$ can be identified with the weight
spaces of $V$. The result follows.
\end{proof}

A $U_q(\mathfrak{b})$-module $V$ is said to be of lowest $\ell$-weight 
$\Psib\in \tb^*_\ell$ if there is $v\in V$ such that $V =U_q(\mathfrak{b})v$ 
and the following hold:
\begin{align*}
U_q(\bo)^- v = \CC v\,,
\qquad 
\phi_{i,m}^+v=\Psi_{i, m}v\quad (i\in I,\ m\ge 0)\,.
\end{align*}
For $\Psib\in\tb^*_\ell$, we have the simple $U_q(\bo)$-module $L'(\Psib)$ of lowest $\ell$-weight $\Psi$, which is not in category $\mathcal{O}^*$ in general. 
We define the fundamental representation $L_{i,a}'^\pm$ whose lowest $\ell$-weight is the
highest $\ell$-weight of $L_{i,a}^\pm$.

The analog of Proposition \ref{rat} holds and we have the notion of characters
and $q$-characters for category $\mathcal{O}^*$ as in Section \ref{qcharo}.

For $V$ a $\tb$-diagonalizable $U_q(\bo)$-module 
with finite-dimensional weight spaces, we have 
$$\chi(V^*) = \chi^{-1}(V)$$ 
where $\chi^{-1}(V)$ is obtained from $\chi(V)$
by replacing each $[\omega]$ by $[\omega^{-1}]$ for $\omega\in\tb^*_\ell$.

\begin{prop}\label{dualweight} Let $V$ be a $U_q(\bo)$-module of lowest $\ell$-weight $\Psib$ with finite-dimensional weight spaces. Then $V^*$ contains a highest $\ell$-weight vector of $\ell$-weight $\Psib^{-1}$.
In the case $V$ irreducible, we get $(L'(\Psib))^* \simeq L(\Psib^{-1})$.
\end{prop}

\begin{proof} 
Let $v$ be a lowest $\ell$-weight vector in $V$. Let $v^*\in V^*$
such that $v^*(v) = 1$ and $v^*$ is zero on higher weight spaces of 
$V$. As $\chi(V^*) = \chi^{-1}(V)$ and the weight of $v^*$ is the opposite
of the weight of $v$, $v^*$ has maximal weight in $V^*$, and so 
$e_j\, v^* = 0$ for any $j\in I$. 
Let us compute the $\ell$-weight $\Psib^* = (\Psi_j^*(z))_{j\in I}$ of $v^*$ 
in $V^*$ in terms of $\Psib = (\Psi_j(z))_{j\in I}$.

As $U_q(\mathfrak{b})$ is a Hopf algebra, the linear morphism 
\begin{align*}
\mathcal{D} :  V \otimes V^* 
\rightarrow 
\CC\,,\quad
\mathcal{D}(u\otimes v) = v(u)
\end{align*}
is a morphism of $U_q(\mathfrak{b})$-module. 

Let $j\in I$. From Theorem \ref{apco}, we have in 
$V \otimes V^*$
$$
\phi_j^+(z)(v\otimes v^*) 
= \phi_j^+(z)v\otimes \phi_j^+(z)v^*
= \Psi_j(z) \Psi_j^*(z)(v\otimes v^*).
$$
As $\mathcal{D}(v\otimes v^*) = 1$ and the action of
$\phi_j^+(z)$ on the trivial module is $1$, by applying $\mathcal{D}$ we get 
$\Psi_j^*(z) = \Psi_j^{-1}(z)$.

For the last point, by construction, $V$ is irreducible if and only if $V^*$ is irreducible.
\end{proof}

\begin{rem}\label{smoins}
If we had used $S$ instead of $S^{-1}$ to define
$V^*$, we should have used a map 
$V^* \otimes V\rightarrow \CC$
instead of $\mathcal{D}$ (as for example in the proof of \cite[Lemma 6.8]{Fre2}). 
But then by using Theorem \ref{apco} as above, we could get additional terms because
$v^*$ has maximal weight and $v$ has minimal weight.\end{rem}

Let $L(m)$ be an irreducible finite-dimensional 
representation of $U_q(\mathfrak{g})$. 
As a $U_q(\bo)$-module, $L(m)$ is in category $\mathcal{O}$
and in category $\mathcal{O}^*$, and is irreducible 
by Proposition \ref{fd}. By \cite[Corollary 6.9]{Fre2},
the lowest weight monomial of $L(m)$ is
$$
m' = \prod_{i\in I, a\in\CC^*}Y_{\overline{i},aq^{r^\vee h^\vee}}^{-u_{i,a}}
$$ 
if we denote $u_{i,a}$ the multiplicity of $Y_{i,a}$ in $m$. 
Here $\overline{i}$ 
is defined so that $w_0(\alpha_i) = - \alpha_{\overline{i}}$ 
where $w_0$ is the longest element of the Weyl group, $h^\vee$ is the 
dual Coxeter number, and $r^\vee$
is the maximal number of edges connecting two vertices of the Dynkin diagram
of $\gb$. 
So we have 
\begin{equation}\label{remkr}
L(m) = L'(m')\text{ and }(L(m))^* \simeq L((m')^{-1}).
\end{equation} 
This last isomorphism is analogous 
to \cite[Proposition 5.1(b)]{Cha2} where the standard duality is used.

\section{Asymptotic representations and $L_{i,1}^-$}\label{asymptminus}

In this section, we shall construct the irreducible $U_q(\bo)$-module
$L^{-}_{i,1}$ as a limit of KR modules of $U_q(\g)$. 
Throughout this section we assume that $|q| > 1$ 
(except when it is mentioned explicitly). 

\subsection{Example}\label{premex}

As an illustration, let us first consider the simplest example 
$U_q(\widehat{\mathfrak{sl}}_2)$. 
Consider the KR modules $L(M_k)$
with 
$M_k=Y_{1,q^{-1}}Y_{1,q^{-3}}\cdots Y_{1,q^{-2k+1}}$. 
Its normalized $q$-character is 
$$
1 + A_{1,1}^{-1} + A_{1,1}^{-1}A_{1,q^{-2}}^{-1} 
+ \cdots + A_{1,1}^{-1} \cdots A_{1,q^{-2(k - 1)}}^{-1}\,.
$$
When $k\rightarrow \infty$,  
it converges as a formal power series in the $A_{1,b}^{-1}$ to
$$
\sum_{j\geq 0} A_{1,1}^{-1}A_{1,q^{-2}}^{-1}\cdots A_{1,q^{-2(j - 1)}}^{-1}.
$$
Let us explain this in terms of representations.

The module $L(M_k)$ carries a basis $(v_0,\cdots,v_k)$ 
with the action of $U_q(\widehat{\mathfrak{sl}}_2)$ given by 
\begin{align*}
&x_{1,r}^+v_j=q^{2r(-j+1)} v_{j-1}\,,
\quad 
x_{1,r}^-v_j=q^{-2rj}[j+1]_q[k - j]_qv_{j+1}\,,
\\
&\phi^{\pm}_1(z)v_j=q^{k-2j}
\frac{(1-q^{-2k}z)(1-q^{2}z)}{(1-q^{-2j+2}z)(1-q^{-2j}z)}v_j\,.
\end{align*}
Observe that the action of the $x_{1,r}^+$ on this basis
does not depend on $k$. 
In contrast, that of the $x_{1,r}^-$ depends on $k$ and diverges 
as $k\to\infty$. Nevertheless the actions of 
$\tilde{x}_{1,r}^-$ and  $\tilde{\phi}_1^\pm(z)$ converge.
Altogether these limiting operators give rise to the following 
`asymptotic' representation of $\widetilde{U}_q(\widehat{\mathfrak{sl}}_2)$ 
on the space $V_\infty = \oplus_{j=0}^\infty \CC v_j$:
\begin{align*}
&\tilde{x}_{1,r}^+v_j = q^{2r(-j+1)} v_{j-1}\,,
\quad
\tilde{x}_{1,r}^-v_j = q^{-2rj+j+2}\frac{[j+1]_q}{q - q^{-1}} v_{j+1}\,,
\\
&\tilde{\phi}_1^+(z)v_j = \frac{1 - q^2z}{(1 - q^{-2j+2}z)(1 - q^{-2j}z)}v_j\,,
\\
&\tilde{\phi}_1^-(z)v_j = 
-q^{4j}\frac{ z^{-1} (1 -q^{-2}z^{-1})}{(1 - q^{2j-2}z^{-1})(1 - q^{2j}z^{-1})}v_j\,.
\end{align*}
Introducing a $Q$-grading  
by $\text{deg}(v_j) = -j\alpha_1$, and defining $k_1v_j=q^{-2j}v_j$, 
we obtain by Proposition \ref{asym} 
a structure of a $U_q(\mathfrak{b})$-module on $V_\infty$.
This representation is clearly irreducible, 
and $\phi_1^+(z)v_0 = (1 - z)^{-1} v_0$. 
So we get the following.

\begin{lem} 
$V_\infty$ is isomorphic to $L_{1,1}^-$ as a $U_q(\mathfrak{b})$-module.
\end{lem}

It is easy to show that this action cannot be extended to 
an action of the full quantum affine algebra 
$U_q(\widehat{\mathfrak{sl}}_2)$.

As another example, 
let us study the KR modules for $U_q(\widehat{\mathfrak{sl}}_3)$ 
$$
V_k=L(M_k),\quad M_k=Y_{1,q^{-1}}Y_{1,q^{-3}}\cdots  Y_{1,q^{-2k+1}}\,.
$$
Then $V_k$ has a basis 
$\{v_{n,n'}\}_{0\leq n'\leq n\leq k}$ 
consisting of $\ell$-weight vectors $v_{n,n'}$ of 
$\ell$-weight 
$M_k \cdot 
A_{1, 1}^{-1}\cdots A_{1,q^{-2n+2}}^{-1}\cdot A_{2,q}^{-1} 
\cdots A_{2,q^{- 2n'+ 3}}^{-1}
$.
The action of the generators is given explicitly as follows.
\begin{align*}
&x_{1,r}^+ v_{n,n'} = q^{r(- 2 n + 2)} [n-n']_q v_{n - 1,n'}\,,
\quad
x_{2,r}^+ v_{n,n'} = q^{r(- 2n'+3)} v_{n,n' - 1},
\\
&x_{1,r}^- v_{n,n'} = q^{-2 nr}[k-n]_q v_{n + 1,n'},
\quad
x_{2,r}^- v_{n,n'} = q^{r(- 2n'+1)} [n'+1]_q [n-n']_q v_{n,n' + 1},
\\
&\phi_1^{\pm}(z)v_{n,n'} = q^{k-2n+ n'} \frac{(1 - q^{-2k}z) (1 -q^{-2n'+2}z)}
{(1-q^{- 2n}z ) (1 -q^{-2n+2}z)} v_{n,n'},
\\
&\phi_2^{\pm}(z)v_{n,n'} = 
q^{n-2n'} \frac{(1 - q^3z) (1 - q^{-2n+1}z)}
{(1 - q^{- 2n'+1}z) (1 -q^{- 2n'+3}z)} v_{n,n'}\,.
\end{align*}
Here we set $v_{n,n'} = 0$ unless $0\leq n'\leq n\leq k$. 
Note that the action of 
$\tilde{x}_{2,m}^-$ and $\tilde{\phi}_2^{\pm}(z)$ do not depend on $k$.
Letting $k\rightarrow \infty$, we get a representation of 
$\widetilde{U}_q(\widehat{\mathfrak{sl}}_3)$ 
on
$V_\infty=\oplus_{0\le n'\le n}\C v_{n,n'}$:
\begin{align*}
&\tilde{x}_{1,r}^+v_{n,n'} 
= q^{r(- 2 n + 2)}[n-n']_q\,  v_{n - 1,n'}\,,
\quad
\tilde{x}_{2,r}^+ v_{n,n'} =  q^{r(- 2n'+3)}\, v_{n,n' - 1}\,,
\\
&\tilde{x}_{1,r}^- v_{n,n'} 
= \frac{q^{n- n'+2-2nr} }{q - q^{-1}}\, v_{n + 1,n'}\,,
\\
&\tilde{x}_{2,r}^- v_{n,n'} = q^{2n' - n +2+ r( - 2n'+1)} [n'+1]_q [n-n']_q\, v_{n,n' + 1}\,,
\\
&\tilde{\phi}_1^+(z)v_{n,n'} 
= \frac{1 -q^{- 2n' + 2}z}{(1 - q^{- 2n}z ) (1 -q^{-2n+2}z)}\, v_{n,n'}\,,
\\
&\tilde{\phi}_1^-(z)v_{n,n'} 
= - q^{4n - 2n'}\frac{z^{-1}(1 - q^{2n' - 2}z^{-1})}{(1 - q^{2n}z^{-1} ) (1 -q^{2n - 2}z^{-1})}\, v_{n,n'}\,,
\\
&\tilde{\phi}_2^{\pm}(z)v_{n,n'}
 = \frac{(1 - q^{3}z) (1 - q^ {-2n+1}z)}{(1 -q^{- 2n'+1}z) (1 -q^{- 2n'+3}z)}\, v_{n,n'}\,.
\end{align*}
These examples have appeared in \cite{BLZ,BHK}
as representations of the Borel algebra, 
but the action of the full asymptotic algebra was not discussed.

\subsection{General case}\label{asymptmoins}

We now proceed to the construction of $L_{i,1}^-$ in the general case. 
Since a direct computational method is hardly applicable in general,
we take an alternative approach. 

Fix $i\in I$, and consider a family of KR modules labeled by $k\ge0$:
\begin{align*}
V_k=L(M_k)\,,
\quad M_k = Y_{i,q_i^{-1}}Y_{i,q_i^{-3}}\cdots Y_{i,q_i^{-2k+1}}\,,
\end{align*}
where we set $M_0=1$.

For each $k,l$ satisfying $k\ge l\ge 0$, we decompose $M_k$
 as 
$M_k^{\ge (-2l+1)d_i}M_k^{<(-2l+1)d_i}$, so that $M_k^{\ge(-2l+1)d_i}=M_l$. 
Let 
\begin{align}
F_{k,l}~:~V_l\longrightarrow V_k^{\ge (-2l+1)d_i}
\label{Fkl}
\end{align}
be the corresponding isomorphism of $U_q(\g)^+$-modules given in Theorem \ref{H3}. 
Fixing a highest weight vector $v_k\in V_k$ we normalize 
\eqref{Fkl} by $F_{k,l}v_l=v_k$. We have then for all $k\ge l\ge m$
\begin{align*}
&F_{k,l}\circ F_{l,m}=F_{k,m}\,,
\quad 
F_{k,k}=\mathrm{id}\,.
\end{align*}
Thus $(\{V_k\},\{F_{k,l}\})$ constitutes an inductive system of $U_q(\g)^+$-modules. 
Set
\begin{align*}
V_\infty=\underset{\longrightarrow}{\lim}\ 
V_k\,,\quad 
v_\infty=F_{\infty,k}\,v_k\,, 
\end{align*}
where 
$F_{\infty,k}: V_k\to V_\infty$
denotes the injective morphism satisfying the condition $F_{\infty,k}\circ F_{k,l}=F_{\infty,l}$. 

It is known \cite{n, hcr} that, in the limit $k\to \infty$,  
the normalized $q$-character $\tilde{\chi}_q(V_k)$
converges
to a formal power series in $\ZZ[[A_{i,q^a}^{-1}]]_{i\in I, a\in\Z}$.
In particular, the dimension of each weight subspace of $V_k$ stabilizes
as $k\to\infty$. 

\begin{prop}\label{cvg} 
For $k\ge l\ge 0$ we have
\begin{align}
&\phi^\pm_i(z)\, F_{k,l}=q_i^{k-l}\frac{1-q_i^{-2k}z}{1-q_i^{-2l}z}\times
F_{k,l}\,\phi^\pm_i(z)\,,
\label{Fkl-phi1}\\
&\phi^\pm_j(z)\, F_{k,l}=F_{k,l}\, \phi^\pm_j(z)\quad (j\neq i)\,.
\label{Fkl-phi2}
\end{align}
For each $l$, the limits 
$\lim_{k\to\infty}F_{k,l}^{-1} \tilde{\phi}^\pm_i(z)  F_{k,l}$ 
exist in $\mathop{\rm End}(V_l)$ and gives rise to an endomorphism of $V_\infty$. 
\end{prop}

\begin{rem}\label{remcvg} 
It follows from Proposition \ref{cvg} 
that the operator $F_{k,l}^{-1} \tilde{\phi}_j^\pm(z) F_{k,l}$ 
is of the form $A(z) + q_i^{-2k} B(z)$ 
where the operators $A(z)$, $B(z)$ do not depend on $k$ and are 
well-defined without any assumption on $|q|$. So the limiting 
operator $A(z)$ on $V_\infty$ makes sense
for such a $q$. 
\end{rem}

\begin{proof}
Formulas \eqref{Fkl-phi1}, \eqref{Fkl-phi2} follow from 
\eqref{Fphi} and 
\begin{align*}
M_k\bigl(\phi^\pm_j(z)\bigr)=
\begin{cases}
q_i^k\displaystyle{\frac{1-q_i^{-2k}z}{1-z}} & (j=i)\,,\\
1 & (j\neq i)\,.
\end{cases}
\end{align*}
Clearly $\lim_{k\to\infty}F_{k,l}^{-1} \tilde{\phi}^\pm_j(z)  F_{k,l}$ exists, and 
\begin{align*}
\bigl(\lim_{k\to\infty}F_{k,l}^{-1} \tilde{\phi}^\pm_j(z)  F_{k,l}\bigr) \circ F_{l,m}
=F_{l,m}\circ\bigl(
\lim_{k\to\infty}F_{k,m}^{-1} \tilde{\phi}^\pm_j(z)  F_{k,m} \bigr)\,
\end{align*}
holds for all $l\ge m$. 
Hence they give rise to a well defined operator 
on $V_\infty$. 
\end{proof}

Next let us study the convergence property for $\tilde{x}^-_{j,r}$.
\begin{lem}
Suppose $k\ge l\ge 0$. Then we have
\begin{align*}
\tilde{x}^{-}_{j,r}V_k^{\ge(-2l+1)d_i}
\subset V_k^{\ge(-2l+1-2\delta_{i,j})d_i}\,
\quad(j\in I,\ r\in\Z).
\end{align*}
\end{lem}
\begin{proof}
Let $v$ be a highest $\ell$-weight vector of $V_k^{<(-2l+1)d_i}$, and let $w\in V_l$. 
From the formula for coproduct in Theorem \ref{apco}
we have
\begin{align*}
&x^{-}_{j,r}(w\otimes v)=\sum_{p=0}^{r-1}x^-_{j,r-p}w \otimes \phi^+_{j,p}v+w\otimes x^-_{j,r}v
\quad (r>0)\,,
\\
&x^{-}_{j,-r}(w\otimes v)=\sum_{p=0}^{r-1}x^-_{j,-r-p} w \otimes \phi^-_{j,-p}v+w\otimes x^-_{j,-r}v
\quad (r\ge0)\,.
\end{align*}
It follows that
\begin{align*}
F_{k,l}\bigl(x^-_{j,r}(w\otimes v)\bigr)\in V_k^{\ge(-2l+1)d_i}+F_{k,l}\bigl(w\otimes x^-_{j,r}v\bigr)\,
\quad (r\in \Z). 
\end{align*}
If $j\neq i$, the second term is absent. If $j=i$, then 
by \cite[Lemma 4.4]{hcr}, $x^-_{i,r}v$ belongs to the generalized 
eigenspace corresponding to the monomial 
$M_k^{<(-2l+1)d_i}A^{-1}_{i,q_i^{-2l-1}}$. 
The assertion of the Lemma follows from these.
\end{proof}

\begin{prop}\label{cvg2} 
For $j\in I$, $r\in\Z$ and $k > l\ge 1$, the operator 
$F_{k,l+1}^{-1} \tilde{x}^-_{j,r} F_{k,l}
\in\mathop{\rm Hom}(V_l,V_{l+1})$ is of the form 
\begin{equation}\label{corform}
F_{k,l+1}^{-1} \tilde{x}^-_{j,r} F_{k,l} = C + q_i^{-2k} D\,,
\end{equation} 
where $C, D\in\mathop{\rm Hom}(V_l,V_{l+1})$
do not depend on $k$. In particular, the limit
\begin{align}
\lim_{k\to\infty}F_{k,l+1}^{-1} \tilde{x}^-_{j,r} F_{k,l}
\in\mathop{\rm Hom}(V_l,V_{l+1})
\label{conv-x}
\end{align}
exists and gives rise to an endomorphism of $V_\infty$. 
\end{prop}

\begin{proof}
Let $\omega_l\in\tb^*$ be the highest weight of $V_l$. 
We prove the convergence on each weight subspace 
$(V_l)_{\omega_l\, \overline{\beta}^{-1}}$ by induction 
on $\beta\in Q$ with respect to the ordering \eqref{partial}.  
If $\beta\not\in Q^+=\oplus_{i\in I}\Z_{\ge0}\alpha_i$, 
then there is nothing to show. 
Suppose the assertion is true for elements of $Q$ smaller than $\beta$, and   
let $v\in (V_l)_{\omega_l\, \overline{\beta}^{-1}}$.
Since $x^+_{j',r'}$ commutes with $F_{k,l}$ for any 
$j'\in I$ and $r'\in \Z$, we have
\begin{align*}
&q_{j'}^{-C_{j',j}}x^+_{j',r'}
\left(F_{k,l+1}^{-1}\tilde{x}^{-}_{j,r} F_{k,l}\right)v
=\left(F_{k,l+1}^{-1}\tilde{x}^{-}_{j,r} F_{k,l}\right)x^+_{j',r'}v
\\
&\quad +\delta_{j',j}\frac{1}{q_j-q_j^{-1}}
F_{k,l+1}^{-1}\left(\tilde{\phi}^{+}_{j,r+r'}-\tilde{\phi}^{-}_{j,r+r'}\right)
 F_{k,l}v\,.
\end{align*}
By the induction hypothesis, the first term in the right hand side 
is of the form (\ref{corform}). 
The operator in second term which is equal to
$$
\delta_{j',j}\frac{1}{q_j-q_j^{-1}}
F_{l+1,l}\left(F_{k,l}^{-1}\left(\tilde{\phi}^{+}_{j,r+r'}
-\tilde{\phi}^{-}_{j,r+r'}\right) F_{k,l}\right)
$$
is also of the form (\ref{corform}) 
due to Remark \ref{remcvg}. Hence 
we have a sequence of vectors 
$w_k=F_{k,l+1}^{-1}\tilde{x}^{-}_{j,r} F_{k,l}v$
in a finite dimensional vector space 
$(V_{l+1})_{\omega_{l+1}\,\beta\ga_j^{-1}}$
such that for any $j',r'$, the vector 
$x^{+}_{j',r'}w_k$ is of the form 
\begin{equation}\label{corformdeux}
x^{+}_{j',r'}w_k = C_{j',r'} + q_i^{-2k} D_{j',r'}
\end{equation} 
where the vectors $C_{j',r'}$, $D_{j',r'}$ do not depend on $k$. 
On the other hand, since $V_{l+1}$ is simple, 
the joint kernel of the $x^+_{j',r'}$ on weight subspaces of 
weight lower than $\omega_{l+1}$
is zero. It follows that the sequence $\{w_k\}$ 
is also of the form (\ref{corformdeux}).
So the operator $F_{k,l+1}^{-1} \tilde{x}^-_{j,r} F_{k,l}$ is of the form (\ref{corform}). The existence
of the limit follows immediately.
The well-defineness on $V_\infty$ 
holds by the same reason as in Proposition \ref{cvg}.
\end{proof}

\begin{rem}\label{remcvg2} By construction, 
the limiting operator $C$ in Proposition \ref{cvg2}
is well-defined 
for all $q$ which is not a root of unity.
\end{rem}

Taking the limit $k\rightarrow \infty$   
we get a structure of $\Utg$-module on $V_\infty$, 
since the relations of the algebra are preserved. 
For example, 
on $V_l$ we have
\begin{align*}
&q_i^{-C_{i,j}}\left(F_{k,l+1}^{-1}\tilde{x}_{i,r}^+F_{k,l+1}\right)
\left(F_{k,l+1}^{-1}\tilde{x}_{j,r'}^-F_{k,l}\right) - 
\left(F_{k,l+1}^{-1}\tilde{x}_{j,r'}^-F_{k,l}\right)
\left(F_{k,l}^{-1}\tilde{x}_{i,r}^+F_{k,l}\right)
\\
&= 
\delta_{i,j}F_{l+1,l}\left(F_{l,k}^{-1}
\frac{\tilde{\phi}^+_{i,r+r'}
-\tilde{\phi}^-_{i,r+r'}}{q_i-q_i^{-1}}F_{k,l}\right)\,,
\end{align*}
and so the relation is satisfied by the asymptotic operators on $V_\infty$.
In particular, 
since $F_{k,l}^{-1} k_i^{-1}  F_{k,l}=q_i^{-k+l}k_i^{-1}$, 
$\kappa_i$ acts as $0$ on $V_\infty$. 

This structure of $\Utg$-module on $V_\infty$ makes sense 
without any assumption on $|q|$.
Indeed, the action of the Borel algebra on KR modules 
is well-defined for such a $q$ : there is a basis of $V_k$ 
such that the coefficients 
of the action of the generators $e_i$, $f_i$, $k_i^{\pm 1}$ 
($0\leq i\leq n$) on $V_k$ are Laurent polynomials in $q$, 
see \cite[Section 4]{cp3}. Although the process of taking 
the limit is not well-defined if $|q|\leq 1$,
it suffices to check that the limiting operators on $V_\infty$ 
make sense at $q$. 
The operators $F_{k,l}^{-1} \tilde{x}_{j,m}^+ F_{k,l}$ are 
constant when $k\rightarrow \infty$, so the corresponding 
limiting operators 
on $V_\infty$ make sense for $q$. We have seen in Remarks 
\ref{remcvg} and \ref{remcvg2} that the limiting operators 
of $F_{k,l}^{-1} \tilde{\phi}^\pm(z) F_{k,l}$ 
and of $F_{k,l}^{-1} \tilde{x}_{j,m}^- F_{k,l}$ make sense for $q$. 
So, for the next results, we do not assume necessarily that $|q| > 1$.

The $\Utg$-module $V_\infty$
has a unique $Q$-grading such that $v_\infty$ has degree $0$.

Applying Proposition \ref{asym} we obtain the following. 

\begin{thm}
The space $V_\infty$ has a structure of a $U_q(\mathfrak{b})$-module
which is in category $\mathcal{O}$. The vector $v_\infty$ satisfies 
$e_j\, v_\infty = 0$ 
for any $j\in I$ and has $\ell$-weight 
\begin{align*} 
\Psib(z)=(1,\cdots,\overset{i-th}
{(1-z)^{-1}},\cdots,1)\,.
\end{align*}
\end{thm}

\begin{cor}\label{Lmin}
The module $L_{i,1}^-$ is in category $\mathcal{O}$.
\end{cor}

\begin{proof}  Let $V_\infty'$
be the submodule of $V_\infty$ generated by $v_\infty$.
Then $V_\infty'$ is in category $\mathcal{O}$ of highest $\ell$-weight $\Psib$.
As $L^-_{i,1}$ is a quotient of $V'_\infty$, we get the result.\end{proof}

\section{The representations $L_{i,1}^+$}\label{asymptplus}

Let $i\in I$, and let   $V_\infty$ be 
the $\tilde{U}_{q^{-1}}(\Glie)$-module 
constructed as in the last section, 
with quantum parameter $q^{-1}$ in place of $q$ 
(we have seen at the end of the
last section that this action is defined without any assumption on $|q|$). 
We use the $Q$-grading $V_\infty=\oplus_{\beta\in Q}(V_{\infty})_\beta$
such that its highest weight vector has degree $0$. 
Then from Proposition \ref{asym}, \eqref{vsigma2}, 
there is a natural structure
of $U_q(\mathfrak{b})$-module on $V_\infty$. This representation is denoted by $V_\infty^\sigma$.
Note that the highest weight vector of $V_\infty$ becomes the lowest weight vector in $V_\infty^\sigma$.
The representation $V_\infty^\sigma$ is in category $\mathcal{O}^*$, and so we have
a structure of $U_q(\mathfrak{b})$-module on the graded dual $(V_\infty^\sigma)^*$ as in 
Section \ref{dualcat}. By Lemma \ref{autredef}, the $U_q(\mathfrak{b})$-module 
$(V_\infty^\sigma)^*$ is in  category $\mathcal{O}$. 

Let $v_\infty^*\in (V_\infty^\sigma)^*$ be defined by $v_\infty^*(v_\infty) = 1$ and $v_\infty^* = 0$
on $\bigoplus_{\omega\neq 0} (V_\infty^{\sigma})_{\omega}$. 
The vector $v_\infty^*$ satisfies 
 $e_j\, v_\infty^* = 0$
for any $j\in I$ and has $\ell$-weight 
\begin{align*} 
\Psib(z)=(1,\cdots,\overset{i-th}
{1-z},\cdots,1)\,.
\end{align*}

We get the following consequence.

\begin{cor}\label{Lpl} The module $L_{i,1}^+$ is in category $\mathcal{O}$.
\end{cor}

Let us look at the example of $\widehat{\mathfrak{sl}}_2$
in more details. 
We have 
$V_\infty = \oplus_{j=0}^\infty \CC v_j$ with the action of 
$\tilde{U}_{q^{-1}}(\widehat{\mathfrak{sl}}_2)$
\begin{align*}
&\tilde{x}_{1,r}^+v_j = q^{2r(j-1)} v_{j-1}\,,
\quad
\tilde{x}_{1,r}^-v_j = - q^{2rj-j-2}\frac{[j+1]_q}{q - q^{-1}} v_{j+1}\,,
\quad k_1v_j=q^{2j}v_j.
\end{align*}
Hence the action of $U_q(\bo)$ on $V_\infty^\sigma$ is given by
\begin{align*}
&e_1\, v_j = - q^{-j}\frac{[j+1]_q}{q - q^{-1}} v_{j+1}\,,
\quad
e_0\, v_j = q^{2(j-1)} v_{j-1}\,,
\quad
k_1 v_j = q^{2j} v_j\,.
\end{align*}
We have $(V_\infty^\sigma)^* = \oplus_{j=0}^\infty \CC v_j^*$, where $(v_j^*)_{j\geq 0}$ 
is a basis dual to $(v_j)_{j\geq 0}$. We have 
\begin{align*}
&e_1\, v_j^* = q^{-3j+3}\frac{[j]_q}{q - q^{-1}} v_{j-1}^*\,,
\quad
e_0\, v_j^* = -q^{4j+2} v_{j+1}^*\,,
\quad
k_1 v_j^* = q^{-2j} v_j^*\,.
\end{align*}
There is a unique basis $(w_j^*)_{j\geq 0}$ of $(V_\infty^\sigma)^*$ such that
$w_0^* = v_0^*$, $w_j^*\in \CC^* v_j^*$ and 
$e_1\, w_j^* = w_{j-1}^*$.
We obtain  
\begin{align*}
e_0\, w_j^* = -q^{j + 2}\frac{[j+1]_q}{q - q^{-1}} w_{j+1}^*\,,
\quad
k_1 w_j^* = q^{-2j} w_j^*\,.
\end{align*}

Let us write the example of 
$\widehat{\mathfrak{sl}}_3$ 
in details. We have $V_\infty^\sigma=\oplus_{0\le n'\le n}\C v_{n,n'}$
with the action of $U_q(\bo)$ given by
\begin{align*}
&e_1 v_{n,n'} =  \frac{q^{-n + n'} }{q^{-1} - q}\,, 
\quad 
e_2 v_{n,n'} =  q^{-2n' + n} [n'+1]_q [n-n']_q\, v_{n,n' + 1}\,,
\\
&e_0 v_{n,n'} = - q^{n+n'-2} v_{n-1,n'-1}\,,\\
&k_1 v_{n,n'} = q^{2n-n'} v_{n,n'}\,,\quad
k_2v_{n,n'} = q^{2n'-n}v_{n,n'}\,,\quad
k_0 v_{n,n'} = q^{-n-n'}v_{n,n'}\,.
\end{align*}
We get the action of $U_q(\bo)$ on 
$(V_\infty^\sigma)^* = \bigoplus_{n,n'\geq 0} \CC v_{n,n'}^*$
\begin{align*}
&e_1 v_{n,n'}^* =\frac{q^{-3n + 2n'+3} }{q - q^{-1}}\, v_{n - 1,n'}^*\,,
\quad
e_2 v_{n,n'}^* = - q^{2(n-2n'+2)} [n']_q [n-n'+1]_q\, v_{n,n' - 1}^*\,,
\\
&e_0 v_{n,n'}^* =  q^{2(n+n'+1)}v_{n+1,n'+1}^*\,,
\\
&k_1 v_{n,n'}^* = q^{-2n+n'} v_{n,n'}^*\,,
\quad
k_2v_{n,n'}^* = q^{-2n'+n}v_{n,n'}^*\,,
\quad
k_0 v_{n,n'}^* = q^{n+n'}v_{n,n'}^*\,.
\end{align*}
There is a unique basis 
$(w_{n,n'}^*)_{n,n'\geq 0}$ of $(V_\infty^\sigma)^*$ such that
$e_1w_{n,n'}^* = [n-n']_q w_{n-1,n'}^*$, $e_2w_{n,n'}^* = w_{n,n'-1}^*$,
$w_{0,0}^* = v_{0,0}^*$ and $w_{n,n'}^* =  \lambda_{n,n'} v_{n,n'}^*$ where $\lambda_{n,n'}\in\CC^*$. We get 
$$\lambda_{n,n'} = \lambda_{n-1,n'}(q^{n-n'} - q^{n'-n}) q^{3n-2n'-3}\text{ , }\lambda_{n,n'}[n']_q[n-n'+1]_q = -\lambda_{n,n'-1}q^{4n'-2n-4}.$$
This implies $e_0w_{n,n'}^* = -[n'+1]_q \frac{q^{n+4}}{q - q^{-1}}w_{n+1,n'+1}^*$.

\section{Irreducibility of asymptotic representations 
and character formulas}\label{irred}

\subsection{Irreducibility of $V_\infty$ and the character of $L_{i,1}^-$}

Let $i\in I$. We recall 
the $U_q(\mathfrak{b})$-module $V_\infty$ 
constructed in Section \ref{asymptmoins}. We have proved that
$L_{i,1}^-$ is a subquotient of $V_\infty$.

\begin{thm}\label{irrasymp} 
The module $V_\infty$ is irreducible and is isomorphic to $L_{i,1}^-$.
In particular 
\begin{align*}
\tilde{\chi}_q(L_{i,1}^-)
=\lim_{k\to\infty}\tilde{\chi}_q\bigl(L(M_k)\bigr)
\end{align*}
as a formal power series in 
$\ZZ[[A_{j,q^r}^{-1}]]_{j\in I ,r\in\ZZ}$. 
\end{thm}

\begin{proof} 
Consider the KR module $L(M_k)$ with 
$M_k = Y_{i,q_i^{-1}}Y_{i,q_i^{-3}}\cdots Y_{i,q_i^{-2k+1}}$, 
viewed as a $U_q(\mathfrak{b})$-module.
In view of Remark \ref{grading-shift}, 
we modify its $Q$-grading so that $k_jv=v$ ($j\in I$)
holds on the highest weight vector $v$. 
Denote the resulting $U_q(\mathfrak{b})$-module by $\tilde{L}(M_k)$.  

We set
\begin{align*}
&
\tilde{\chi}_q(\tilde{L}(M_k)) = \sum_{\Psibs\in\mfr}
 n_{\Psibs}^{(k)} [\Psib]\,,
\quad
\tilde{\chi}_q(V_\infty) = \sum_{\Psibs\in\mfr} 
n_{\Psibs}^{(\infty)} [\Psib]\,,
\quad
\tilde{\chi}_q(L_{i,1}^-) = \sum_{\Psibs\in\mfr} 
n_{\Psibs}^- [\Psib]\,.
\end{align*}
Both $\tilde{\chi}_q(\tilde{L}(M_k))$, $\tilde{\chi}_q(V_\infty)$
belong to $\ZZ[[A_{j,b}^{-1}]]_{j\in I, b\in\CC^*}$, 
and $\tilde{\chi}_q(\tilde{L}(M_k))$ 
converges to $\tilde{\chi}_q(V_\infty)$ as $k\to \infty$.
Since $L_{i,1}^-$ is a subquotient of $V_\infty$, 
we have 
\begin{align*}
n_{\Psibs}^{-}\le n_{\Psibs}^{(\infty)}\,
\quad (\Psib\in\mfr).
\end{align*}

To prove the theorem, it suffices to show the reverse inequality.
Fix $\Psib\in\mfr$.   
We may assume $n^{(\infty)}_{\Psibs}\neq 0$.
Comparing the highest $\ell$-weights we see 
that $\tilde{L}(M_k)$ is a subquotient of 
$L^+_{i,q_i^{-2k}}\otimes L^-_{i,1}$. Hence we must have 
\begin{align}
\Psib(z)=\Psib^+_k(q_i^{-2k}z)\Psib^-_k(z)\,
\label{Psibpm}
\end{align}
for some $\Psib^\pm_k(z)$ which occur in $\chi_q(L^\pm_{i,1})$. 
We show that we must necessarily have
$\Psib^+_k(z)=1$ for sufficiently large $k$. 
Indeed, if we write
$\varpi\bigl(\Psib(z)\bigr)=
(\overline{\alpha}_{i_1}\cdots \overline{\alpha}_{i_l})^{-1}$, 
$\varpi\bigl(\Psib^-_k(z)\bigr)=
(\overline{\alpha}_{j_1}\cdots \overline{\alpha}_{j_m})^{-1}$, 
then from 
\begin{align*}
\varpi\bigl(\Psib(z)\bigr)=
\varpi\bigl(\Psib^+_k(q_i^{-2k}z)\bigr)\varpi\bigl(\Psib^-_k(z)\bigr)
\end{align*} 
we have $m\le l$. 
Given $\Psib$, there are only finitely many
possibilities for such $\Psib^-_k(z)$'s. 
Since \eqref{Psibpm} holds for any $k$, $\Psib^+_k(z)$ must be 
independent of $z$ for $k$ large enough.  
It follows that
$(\varpi(\Psib))^{-1} \Psib=(\varpi(\Psib_k^-))^{-1}\Psib_k^-$.
On the other hand, $\Psib,\Psib^-_k$ are both monomials in 
$A_{j,b}^{-1}$'s. Therefore we must have 
$\Psib=\Psib_k^-$ and $\Psib^+_k=1$. 

The multiplicity of the term $[1]$ in 
$\tilde{\chi}_q(L^+_{i,q_i^{-2k}})$ 
is $1$. We conclude that $n^{(k)}_{\Psibs}\le n^-_{\Psibs}$, 
which shows the opposite inequality
\begin{align*}
n_{\Psibs}^{-}\ge n_{\Psibs}^{(\infty)}\,
\quad (\Psib\in\mfr).
\end{align*}
\end{proof}

\subsection{Irreducibility of $(V_\infty^\sigma)^*$ 
and the character of $L_{i,1}^+$}

\begin{prop}\label{classprim} For $\Psib\in\tb^*_\ell$, 
the simple module $L'(\Psib)$ is in category 
$\mathcal{O}^*$ if and only if $\Psi_i(z)$ 
is rational for any $i\in I$.
\end{prop}

\begin{proof} As for Theorem \ref{class}, it suffices to prove that for $i\in I$, the fundamental
representations $L_{i,1}'^\pm$ are in category $\mathcal{O}^*$. Theorem \ref{class} is already proved,
so the representations $L_{i,1}^\pm$ of $U_{q^{-1}}(\bo)$ are 
in category $\mathcal{O}$. 
Since $(L_{i,1}'^\pm)^*\simeq L_{i,1}^\mp$
by Proposition \ref{dualweight}, we get the result.
\end{proof}

We define a partial ordering $\preceq$ on characters : $\chi\preceq \chi'$ if the multiplicities of weights
are lower in $\chi$ than in $\chi'$.

\begin{thm}\label{plusexpl} $(V_\infty^\sigma)^*$ is irreducible isomorphic to $L_{i,1}^+$ and we have $\chi(L_{i,1}^+) = \chi(L_{i,1}^-)$.
\end{thm}

\begin{proof} 
We have proved that $L_{i,1}^+$ is a subquotient of $(V_\infty^\sigma)^*$ and
$\chi((V_\infty^\sigma)^*) = \chi(V_\infty) = \chi(L_{i,1}^-)$.
So $\chi(L_{i,1}^+)\preceq \chi(L_{i,1}^-)$.

To prove the reverse inequality, let $k\geq 1$ and consider the KR module $L'(P_k^{-1})$ with 
$P_k = Y_{i,q_i}Y_{i,q_i^3}\cdots Y_{i,q_i^{2k+1}}$ (this is a KR module by (\ref{remkr})).
This is an irreducible $U_q(\bo)$-module by Proposition \ref{fd}. 
We modify its $Q$-grading so that $k_jv_k = v_k$
($j\in I$) holds on the 
lowest weight vector $v_k$. 
In the resulting module $\tilde{L}'(P_k^{-1})$, we have
$$
\phi_i^+(z)\,v_k = \frac{1 - z q_i^{2(k+1)}}{1 - z}\,v_k\,,
\quad
\phi_j^+(z)\,v_k = v_k\quad \text{ for }j\neq i\,.
$$ 
So $\tilde{L}'(P_k^{-1})$ is a subquotient of $L'^-_{i,1}\otimes L'^+_{i,q_i^{2(k+1)}}$.

By twisting by $\sigma$, the $q$-character of $L'(P_k^{-1})$ is equal
to the $q$-character of the $U_{q^{-1}}(\mathfrak{g})$-KR module 
$L(Y_{i,(q_i^{-1})^{-2k-1}}Y_{i,(q_i^{-1})^{-2k+1}}\cdots Y_{i,(q_i^{-1})^{-1}})$.
Hence, both $P_k \chi_q(\tilde{L}'(P_k^{-1}))$ and $\Psib^{-1}\chi_q(V_\infty^\sigma)$ belong to 
$\ZZ[[A_{j,b}]]_{j\in I,b\in\CC^*}$ ($\Psib$ is the $\ell$-weight of $V_\infty^\sigma$ of minimal weight), 
and 
$$\Psib^{-1}\chi_q(V_\infty^\sigma) = \text{lim}_{k\rightarrow \infty} P_k \chi_q(\tilde{L}'(P_k^{-1}))$$ 
as a formal power series in $\ZZ[[A_{j,b}]]_{j\in I,b\in\CC^*}$.

Now we can conclude as for Theorem \ref{irrasymp} that $\chi(L'^-_{i,1}) \succeq \chi^{-1}(L_{i,1}^-)$.

So we have $\chi(L_{i,1}^+) = \chi^{-1}(L_{i,1}'^-) \succeq \chi(L_{i,1}^-)$. This implies the result.
\end{proof}

\subsection{Explicit character formulas for fundamental representations.}

While the convergence of the normalized
$q$-characters has been proven \cite{n, hcr},  
no explicit formula for the limit is known in general. 
For the ordinary characters $\chi\bigl(L(M_k)\bigr)$,  
explicit formulas are known \cite{n, hcr}, 
from which one can extract the following formula for the limit.

\begin{thm} For any $i\in I$, $a,b \in\CC^*$, we have
$$
\chi(L_{i,a}^+) = \chi(L_{i,b}^-) = 
\frac{\underset{N=(N_k^{(j)})}{\sum}
\underset{j\in I, k>0}{\prod}
\begin{pmatrix}
N_k^{(j)} + \delta_{i,j} k - 
\underset{h\in I,l>0}{\sum}N_l^{(h)}r_jC_{j,h}
\text{min}(\frac{k}{r_h},\frac{l}{r_j})
\\N_k^{(j)}\end{pmatrix}[\overline{\alpha_j}]^{-k
N_k^{(i)}}}{\underset{\alpha\in\Delta_+}{\prod}(1-[\overline{\alpha}]^{-1})}
$$ 
where $\Delta_+$ is the set of positive roots of 
$\gb$,
$\binom{a}{b}
=\Gamma(a+1)/\bigl(\Gamma(a-b+1)\Gamma(b+1)\bigr)$, and the sum is taken over all non-negative integers
$N^{(j)}_k$ with $j\in I$, $k\in \Z_{>0}$.
\end{thm}

We have proved an explicit character formula for all fundamental representations in category $\mathcal{O}$.

\section{Asymptotic representations and $L_{i,1}^+$}\label{asymptpplus}

In this section we study another limit of the KR modules and discuss its relation 
to the module $L^+_{i,1}$. We assume $|q| < 1$ throughout.

\subsection{First examples}

We begin with the simplest example of $U_q(\widehat{\mathfrak{sl}}_2)$. 
Consider the KR module 
$W_k=L(N_k)$
with highest monomial 
$N_k=Y_{1,q} Y_{1,q^3}\cdots Y_{1,q^{2k-1}}$. 
It has a basis $(v_0,\cdots,v_k)$ with the action 
\begin{align*}
&x_{1,r}^+v_j=q^{2r(k-j+1)} v_{j-1}\,,
\quad 
x_{1,r}^-v_j 
= q^{2r(k-j)}[k-j]_q[j+1]_qv_{j+1}\,,
\\
&\phi_1^\pm(z)v_j= q^{k-2j}
\frac{(1-z)(1-zq^{2(k+1)})}{(1-zq^{2(k-j+1)})(1-zq^{2(k-j)})}v_j\,.
\end{align*}
Unlike the case of $L^-_{1,1}$ discussed in the previous section, 
only a `half' of these operators converge as $k\to\infty$.
\begin{align*}
&\lim_{k\to\infty}\tilde{x}_{1,r}^+v_j=\delta_{r,0} v_{j-1}\quad
(r\ge 0)\,,
\\
&\lim_{k\to\infty}\tilde{x}_{1,p}^-v_j=\frac{-q^{j+2}\delta_{p,1}}{q-q^{-1}}[j+1]_qv_{j+1}\quad
(p\ge 1)\,,
\\
&\lim_{k\to\infty}\tilde{\phi}_1^+(z)v_j=(1-z)v_j\,.
\end{align*}
By setting $k_1v_j =  q^{-2j}v_j$, we get an action of $U_q(\mathfrak{b})$ on 
$W_\infty=\oplus_{j\geq 0} \CC v_j$ :
\begin{align*}
&x_{1,r}^+v_j = \delta_{r,0} v_{j-1}\,,\quad 
x_{1,p}^-v_j = \frac{- q^{-j}\delta_{p,1}}{q-q^{-1}}[j+1]_q v_{j+1}\,,
\\
&\phi_1^+(z)v_j = q^{-2j}(1-z)v_j.
\end{align*}
This representation is simple and isomorphic to 
$L_{1,1}^+$. We recover
the action of the example of the last section
\begin{align*}
e_1\,v_j = v_{j-1}\,, \quad e_0\, v_j = -q^{j + 2}\frac{[j+1]_q}{q - q^{-1}} v_{j+1}\,,
\quad
k_1 v_j = q^{-2j} v_j\,.
\end{align*}
It is easy to check that this action cannot be extended 
to that of the quantum affine algebra $U_q(\widehat{\mathfrak{sl}}_2)$.

We note that, in contrast to the case $L^-_{1,1}$, the normalized $q$-character 
\begin{align*}
\tilde{\chi}_q\bigl(L_{1,1}^+\bigr)=\sum_{j=0}^\infty[q^{-2j}]
\end{align*}
is independent of $z$ and is {\it not} a formal power series in 
$A_{1,a}^{-1}$'s.

Let us study another example for $U_q\bigl(\widehat{\mathfrak{sl}}_3\bigr)$.
Consider the KR modules 
\begin{align*}
W_k=L(N_k),\quad N_k= Y_{1,q}Y_{1,q^3}\cdots Y_{1,q^{2k-1}}.
\end{align*}
It has a basis $\{v_{n,n'}\}_{0\leq n'\leq n\leq k}$ 
with the action 
\begin{align*}
&\tilde{x}_{1,m}^+ v_{n,n'} 
= q^{2m(1 + k - n)} [n-n']_q \,v_{n - 1,n'}\,,
\quad
\tilde{x}_{2,m}^+ v_{n,n'} = q^{m(3 + 2k - 2n')} \, v_{n,n' - 1}\,,
\\
&\tilde{x}_{1,p}^- v_{n,n'} = 
q^{-k +2n - n'+2 + 2p(k - n)} 
[k-n]_q \, v_{n + 1,n'},
\\
&\tilde{x}_{2,p}^- v_{n,n'} = q^{- n + 2n' +2+ p(1 + 2k - 2n')} [n'+1]_q [n-n']_q \, v_{n,n' + 1},
\\
&\tilde{\phi}_1^{\pm}(z)v_{n,n'} =  
\frac{(1 - z) (1 - zq^{2(1 + k - n')})}{(1 - z q^{2(k - n)} ) (1 - zq^{2(1+k-n)})} \, v_{n,n'},
\\
&\tilde{\phi}_2^{\pm}(z)v_{n,n'} =  
\frac{(1 - zq^{1 + 2k - 2n}) (1 - zq^{3 + 2k})}{(1 - z q^{1 + 2k - 2n'}) (1 - zq^{3 + 2k - 2n'})}
\, v_{n,n'}\,.
\end{align*}
For $m\geq 0$ and $p > 0$, these operators converge when $k\rightarrow \infty$,
\begin{align*}
&\tilde{x}_{1,m}^+ v_{n,n'} = \delta_{m,0}[n-n']_q\,  v_{n - 1,n'}\,,
\quad
\tilde{x}_{1,p}^- v_{n,n'} = \delta_{p,1}
\frac{- q^{n - n'+2}}{q - q^{-1}}\, v_{n + 1,n'},
\\
&\tilde{x}_{2,m}^+ v_{n,n'} = \delta_{m,0}\,  v_{n,n' - 1}\,,
\quad\tilde{x}_{2,p}^- v_{n,n'} = 0\,,
\\
&\tilde{\phi}_1^+(z) v_{n,n'} =  (1 - z)  v_{n,n'}\,,
\quad \tilde{\phi}_2^+(z)v_{n,n'} =  v_{n,n'}\,.
\end{align*}
In addition, the operator
$e_0 = \tilde{x}_{2,1}^- \tilde{x}_{1,0}^-  - q \tilde{x}_{1,0}^- \tilde{x}_{2,1}^-$ 
also converges since
\begin{align*}
e_0\, v_{n,n'}& =[n'+1]_q [k-n]_q q^{k+4}
\, v_{n+1,n'+1}\,
\quad\longrightarrow
\quad 
[n'+1]_q \frac{-q^{n+4}}{q - q^{-1}}  
\, v_{n+1,n'+1}.
\end{align*}
In particular we get an  
asymptotic 
action of the Borel algebra
on $W_\infty=\oplus_{0\le n'\le n}\C v_{n,n'}$.
The action of $\tilde{x}_{1,0}^-$ does not converge 
but the action of $\tilde{x}_{2,0}^-$ is constant.
This example appeared in \cite{BHK}.

\subsection{First approach}\label{firsta}

From now on, we shall be concerned with the family of KR modules of 
$U_q(\g)$
\begin{align*}
W_k=L(N_k),\quad 
N_k = Y_{i,q_i^{2k - 1}}\cdots Y_{i,q_i^3}Y_{i,q_i}\,,
\end{align*}
with the convention $N_{0} = 1$. 
More generally we consider for $k\ge l\ge 0$ the modules
$W_{k,l}=L(N_{k,l})$ with 
$N_{k,l}=Y_{i,q_i^{2k-1}}\cdots Y_{i,q_i^{2k-2l+1}}$. 
We fix a highest weight vector $w_{k,l}\in W_{k,l}$ and write $w_k=w_{k,k}$. 

We have a unique isomorphism of vector spaces
\begin{align*}
&H_{k,l}~:~W_l\longrightarrow W_{k,l},
\end{align*}
such that $H_{k,l}w_l=w_{k,l}$ and 
\begin{align}
&x\,H_{k,l}=H_{k,l}\, \tau_{q_i^{2(k-l)}}(x)\quad
(x\in U_q(\g))\,, 
\label{taua}
\end{align}
where $\tau_a$ denotes the automorphism of $U_q(\g)$ given in  \eqref{tau-a}. 

Decomposing the monomial as $N_k=N_{k,l}N_{k-l}$, 
we consider 
the corresponding morphism of $U_q(\g)^+$-modules in Theorem \ref{H3}
\begin{align*}
&G_{k,l}~:~ W_{k,l}\longrightarrow W_k^{\ge(2k-2l+1)d_i}\,,
\end{align*}
normalized as $G_{k,l}w_{k,l}=w_k$. 
Set further
\begin{align*}
I_{k,l}=G_{k,l}\circ H_{k,l}~:~W_l\longrightarrow W_k^{\ge(2k-2l+1)d_i}.
\end{align*}
Clearly
\begin{align*}
&I_{k,l}\circ I_{l,m}=I_{k,m}\quad (k\ge l\ge m),\quad I_{k,k}=\mathrm{id},
\end{align*}
so that $(\{W_k\},\{I_{k,l}\})$ constitutes an inductive system of linear spaces. 
Let 
\begin{align*}
W_\infty=\underset{\longrightarrow}{\lim}W_k,\quad 
w_\infty=I_{\infty,k}w_k, 
\end{align*}
where $I_{\infty,k}: W_k\to W_\infty$
denotes the injective linear map
satisfying the condition $I_{\infty,k}\circ I_{k,l}=I_{\infty,l}$. 

Combining \eqref{taua} with \eqref{Fx}, \eqref{Fphi} we find that 
\begin{align}
&x^+_{j,r}I_{k,l}=q_i^{2(k-l)r}I_{k,l}x^+_{j,r}\,,
\label{Ix}\\
&\phi^{\pm}_j(z)I_{k,l}=
I_{k,l}\phi^{\pm}_j(zq_i^{2(k-l)})\times
\begin{cases}
q_i^{k-l}\frac{1-z}{1-q_i^{2(k-l)}z}& (j=i),\\
1 & (j\neq i).\\
\end{cases}
\label{Iphi}
\end{align}
In particular, $k_j$ is constant if $j\neq i$ and we have
$k_i I_{k,l}=q_i^{k-l}I_{k,l}k_i$.

\begin{prop}\label{limhphi} Consider the limit $k\to\infty$.  

(i) For $r\ge 1$, the operator $I_{k,l}^{-1}x^+_{j,r}I_{k,l}$ converges to $0$. 
For $r=0$ it stays constant. 

(ii) The operator  $I_{k,l}^{-1}\tilde{\phi}_j^+(z)I_{k,l}$ converges to $1-\delta_{i,j}z$.
\end{prop}
\begin{proof}
This follows from \eqref{Ix} and \eqref{Iphi}.
\end{proof}

\begin{prop} Let $j\in I\setminus\{i\}$ and $r\in\ZZ$. The action of the operator 
\begin{align*}
I_{k,l+1}^{-1} \bigl(q_i^{-2kr}\tilde{x}_{j,r}^- \bigr)I_{k,l}
\ \in \mathrm{Hom}(W_l,W_{l+1})
\end{align*}
stays constant. In particular, $I_{k,l+1}^{-1} \tilde{x}_{j,r}^- I_{k,l}$
converges to $0$ if $r \geq 1$ and stays constant if $r = 0$.
\end{prop}

\begin{proof}
We adapt the argument in the proof of Proposition \ref{cvg}.
Let $\omega_l\in\tb^*$ be the highest weight of $W_l$. 
We prove by induction on $\beta\in Q$ 
that when $k\rightarrow \infty$ 
the operator $I_{k,l+1}^{-1}\bigl(q_i^{-2kr}\tilde{x}_{j,r}^- \bigr) I_{k,l}$ 
stays constant on $(W_l)_{\omega_l\overline{\beta}^{-1}}$.  
When $\beta\notin Q^+$ this is clear as $(W_l)_{\omega_l\overline{\beta}^{-1}} = 0$. 
For an arbitrary $\beta$, 
it suffices to prove that
$\tilde{x}_{j',r'}^+ I_{k,l+1}^{-1}\bigl( q_i^{-2kr}\tilde{x}_{j,r}^- \bigr)I_{k,l}$
is constant on $(W_L)_{\omega_l\overline{\beta}^{-1}}$
for all $j'\in I$ and $r'\geq 0$. Furthermore, by the argument of the proof of Proposition \ref{fd}, 
we may assume without loss of generality that $r' > r$. 

We have
\begin{align*}
\tilde{x}_{j',r'}^+ \left(I_{k,l+1}^{-1} q_i^{-2kr}
\tilde{x}_{j,r}^- I_{k,l}\right)
&= q_j^2 \left(I_{k,l+1}^{-1} q_i^{-2kr}\tilde{x}_{j,r}^- I_{k,l}\right)
\tilde{x}_{j',r'}^+ \\
&+\delta_{j,j'} \frac{q_j^2 q_i^{-2kr + 2r'(l-k)}
\left(I_{k,l+1}^{-1}\tilde{\phi}_{j,r+r'}^+ 
I_{k,l}\right)}{q_j - q_j^{- 1}}\,.
\end{align*}
In the right hand side,  
the first term is constant by the induction hypothesis.
The second term equals up to a constant multiple
\begin{align*}
\delta_{j,j'} q_i^{2r'(l-k)-2kr}\left(I_{k,l+1}^{-1}\tilde{\phi}_{j,r+r'}^+ 
I_{k,l}\right) = \delta_{j,j'}q_i^{-2 l r} \tilde{\phi}_{j,r+r'}^+\,,
\end{align*}
which is constant. 
\end{proof}

\begin{prop} Let $r\in\ZZ$. We have
\begin{align*}
I_{k,l+1}^{-1} \bigl(q_i^{2k(1 - r)}\tilde{x}_{i,r}^- \bigr)I_{k,l}
\ = A_r + q_i^{2k} B_r\in \mathrm{Hom}(W_l,W_{l+1})
\end{align*}
where $A_r, B_r\in \mathrm{Hom}(W_l,W_{l+1})$ are constant operators.

In particular, $I_{k,l+1}^{-1} \tilde{x}_{i,r}^- I_{k,l}$
converges to $0$ if $r\geq 2$ and converges if $r = 1$.
\end{prop}

\begin{proof} 
The proof is analogous to the proof of the previous proposition and
we retain the notation there. 
We prove by induction on $\beta\in Q$ 
that the operator $I_{k,l+1}^{-1}\tilde{x}_{i,r}^- I_{k,l}$ 
is of the form $A_r + q_i^{2k} B_r$ on $(W_l)_{\omega_l\overline{\beta}^{-1}}$.  
When $\beta\notin Q^+$ this is clear. For an arbitrary $\beta$, 
it suffices to prove that
$\tilde{x}_{j',r'}^+ I_{k,l+1}^{-1}\bigl( q_i^{2k(1 - r)}\tilde{x}_{i,r}^- \bigr)I_{k,l}$
is of this form on $(W_L)_{\omega_l\overline{\beta}^{-1}}$
for all $j'\in I$, $r'\geq 0$.  
We may assume that $r' > r$ and we have
\begin{align*}
&\tilde{x}_{j',r'}^+ \left(I_{k,l+1}^{-1} q_i^{2k(1 - r)}
\tilde{x}_{i,r}^- I_{k,l}\right)
= q_i^2 \left(I_{k,l+1}^{-1} q_i^{2k(1 - r)}\tilde{x}_{i,r}^- I_{k,l}\right)
\tilde{x}_{j',r'}^+ \\
&+\delta_{i,j'} q_i^2  \frac{q_i^{2k(1 - r)}q_i^{2r'(l-k)}
\left(I_{k,l+1}^{-1}\tilde{\phi}_{i,r+r'}^+ 
I_{k,l}\right)}{q_i - q_i^{- 1}}.
\end{align*}
The first term in 
the right side is of the correct form from the induction hypothesis.
The second term for $i=j'$ is equal, up to a constant multiple, to
\begin{align*}
q_i^{2k(1 - r) + 2r'(l-k)}I_{k,l+1}^{-1}\tilde{\phi}_{i,r+r'}^+ 
I_{k,l} 
&= 
\left[\frac{q_i^{-2 lr}(q_i^{2k}-zq_i^{2l})\tilde{\phi}_i^+(z)}
{1-z}
\right]_{r + r'}
\\
&=\left[
\frac{-zq_i^{2l(1-r)}\tilde{\phi}_i^+(z)}
{1-z}
\right]_{r + r'}
+q_i^{2k-2lr}  
\left[
\frac{\tilde{\phi}_i^+(z)}
{1-z}
\right]_{r + r'}\,,
\end{align*}
where we write $c_n=[f(z)]_n$ for a formal power series $f(z)=\sum_{n=0}^\infty c_nz^n$.  
This is also in the correct form.
\end{proof}

Let $L\geq 1$ be the length of the maximal root of $\gb$
and $\cN_i\geq 1$ 
be the multiplicity of $\alpha_i$ in the maximal root of $\gb$.

\begin{prop}We have
\begin{align*}
I_{k,l+L}^{-1} e_0\, 
I_{k,l}\ = q_i^{2k(1-\cN_i)}  
\sum_{0\leq p\leq N}q_i^{2kp}A_p\in \mathrm{Hom}(W_l,W_{l+L})
\end{align*}
where for $0\leq p\leq N$, $A_p\in \mathrm{Hom}(W_l,W_{l+L})$ is a constant operator.

In particular,  
$I_{k,l+L}^{-1} e_0\, I_{k,l}$ 
is convergent if $\cN_i = 1$.
\end{prop}

\begin{proof} Let us write  $e_0$  
as in (\ref{e0}).  Then it is clear that
for any $k\geq l+L$, the operator $I_{k,l+L}^{-1}x_0^+I_k$ makes sense in $\text{Hom}(W_l,W_{l+L})$.
Then the result follows immediately from the last propositions.
\end{proof}

From here until the end of this 
section \ref{firsta}, 
we assume that $\cN_i = 1$.
We get a structure of $U_q(\mathfrak{b})$-module 
on $W_\infty$ in category $\mathcal{O}$
(with the natural $Q$-grading such that the highest
weight vector has degree $0$).

\begin{thm} $W_\infty$ is irreducible isomorphic to $L_{i,1}^+$ and we have $\tilde{\chi}_q(L_{i,1}^+) = \chi(L_{i,1}^+)$.
\end{thm}
In this case, we get another proof that $L_{i,1}^+$ is in category $\mathcal{O}$.
We also get an explicit $q$-character formula for $L_{i,1}^+$.

\begin{proof} The representation $L_{i,1}^+$ is a subquotient of $W_\infty$.
For any $k\geq 0$, we have $\tilde{\chi}(W_k) = \tilde{\chi}(V_k)$. 
By construction,
$\chi(V_\infty)$ (resp. $\chi(W_\infty)$) 
is the limit of the $\tilde{\chi}(V_k)$ 
(resp. of the $\tilde{\chi}(W_k)$).
Hence $\chi(W_\infty) = \chi(V_\infty)$ which is equal to $\chi(L_{i,1}^+)$ by Theorem \ref{plusexpl}.
The result follows. For the second point, it follows from Proposition \ref{limhphi} that 
$\tilde{\chi}_q(W_\infty) = \chi(W_\infty)$.\end{proof}

Let $W_k'\subset W_k$ be the sum of the $U_q(\gb)$-submodules of $W_k$ which
do not contain the highest weight vector. 
For any $j\in I$, $k\geq l$, we have $x_{j,0}^+ I_{k,l} = I_{k,l} x_{j,0}^+$. 
In particular we have $I_{k,l}(W_l')\subset W_k'$.
Let us consider $W_\infty' = \bigcup_{k\geq 0} I_{\infty,k}(W_k')$.

\begin{lem} We have $W_\infty' = \{0\}$.\end{lem}
\begin{proof} First let us prove that $U_q(\mathfrak{b})W_\infty'$ is a submodule of $W_\infty$. 
The subspace $W_\infty'$ is stable under the action of the 
$x_{j,0}^+, k_j^{\pm 1}$ for $j\in I$. 
Since $\phi^+_{j,1}$ acts as a scalar, 
it follows that $W_\infty'$ is stable also by $x^+_{j,r}$ with $r>0$.
Hence $U_q(\mathfrak{b})^+W_\infty'\subset W_\infty'$
and we have 
\begin{align*}
U_q(\mathfrak{b})W_\infty' = U_q(\mathfrak{b})^-U_q(\mathfrak{b})^0W_\infty'.
\end{align*}
Consequently, for reasons of weights, $w_\infty \notin U_q(\mathfrak{b})W_\infty'$. 
Since $W_\infty$ is irreducible, 
we obtain the assertion.
\end{proof}
As a consequence, we get a second explicit formula for $\chi(L_{i,1}^+)$ in this case :
it is the limit of characters given by the Weyl character formula.

\begin{rem} We also get that for any $k\geq 0, a\in\CC^*$, the KR module
 $L(Y_{i,a}Y_{i,aq_i^2}\cdots Y_{i,aq_i^{2(k-1)}})$ with $\mathcal{N}_i=1$
is irreducible as a representation of $U_q(\gb)$. 
We recover the result of \cite{c1}.
\end{rem}

\subsection{Another example}

Let us look at the $\widehat{\mathfrak{sl}}_2$-case from a different angle. 
We can choose a basis $(v_0',v_1',\cdots, v_k')$ of $W_k$ such that
\begin{align*}
&x_{1,r}^+v_j'=q^{2r(k-j+1) + 2k - 4(j-1)} [j]_q [k-j+1]_q v_{j-1}'\,,
\quad 
x_{1,r}^-v_j' 
= q^{2r(k-j)+4j - 2k} v_{j+1}'\,,
\\
&\phi_1^\pm(z)v_j'= q^{k-2j}
\frac{(1-z)(1-zq^{2(k+1)})}{(1-zq^{2(k-j+1)})(1-zq^{2(k-j)})}v_j'\,.
\end{align*}
In this basis we see that $(f_1k_1^2)v_j' = v_{j+1}'$ and 
\begin{align*}
\lim_{k\to\infty}(e_1k_1^{-1})v_j'
 &=\lim_{k\to\infty} q^{k - 2j + 4} [j]_q [k-j+1]_q v_{j-1}'
\\
&=- [j]_q \frac{q^{-j+3}}{q - q^{-1}}[k-j+1]_q v_{j-1}'.
\end{align*}
Then $(e_0k_0^{-1})v_j' = q^{2(j+1) } v_{j+1}'$ is 
constant and $\tilde{\phi}_1^+(z)v_j'$ 
converges to $(1 - z)v_j'$.

Let us generalize this calculation in the next subsection.

\subsection{Second approach}

For $k\geq 0$, let $V_k$ be
as in Section \ref{asymptmoins} but 
for the parameter $q^{-1}$,
that is, $V_k$ is the 
$U_{q^{-1}}(\Glie)$-module
$L(Y_{i,(q_i^{-1})^{-1}}Y_{i,(q_i^{-1})^{-3}}\cdots Y_{i,(q_i^{-1})^{-(2k-1)}})$. 
Let $v_k$ be the highest $\ell$-weight vector of $V_k$.
Pulling back by $\sigma$, we get the $U_q(\Glie)$-module $V_k^{\sigma}$. 
We have in $V_k^\sigma$
$$
\phi_i^+(z)\,v_k = q_i^{-k} 
\frac{1 - z q_i^{2k}}{1 - z}\,v_k\,,
\quad
\phi_j^+(z)\,v_k = 1\text{ for }j\neq i\,.
$$ 
So the $\ell$-weight of $v_k$ in $V_k^{\sigma}$ is
$Y_{i,q_i}^{-1}Y_{i,q_i^3}^{-1}\cdots Y_{i,q_i^{2k-1}}^{-1} =
N_k^{-1}$. 
We define its dual module $(V_k^\sigma)^*$ as 
in Section \ref{dualcat}. By Proposition \ref{dualweight},
we have a highest $\ell$-weight vector $v_k^*$ in $(V_k^\sigma)^*$ 
of $\ell$-weight $N_k$.
In particular we can identify $(V_k^\sigma)^*$
with $W_k$ and $v_k$ with $w_k$. 

The injective morphism of $U_q(\Glie)^+$-module $F_{k,l} : V_l\rightarrow V_k$ ($1\leq l\leq k$)
gives rise to an injective linear morphism $F_{k,l}^\sigma : V_l^\sigma \rightarrow V_k^\sigma$
and to its surjective 
dual $\Pi_{l,k}=(F_{k,l}^\sigma)^* : W_k \rightarrow W_l$.
Then $(\{W_k\}, \{\Pi_{l,k}\})$ constitutes a projective 
system of vector spaces. Set
\begin{align*}
W_\infty^\Pi=\underset{\longleftarrow}{\lim}\ W_k,
\end{align*}
and let $w_\infty\in W^\Pi_\infty$ be the unique vector satisfying 
$w_k = \Pi_{k,\infty}\, w_\infty^\Pi$,
where $\Pi_{k,\infty} : W_\infty^\Pi \rightarrow W_k^\Pi$ 
denotes the surjective linear morphism 
satisfying $\Pi_{l,k}\circ \Pi_{k,\infty} = \Pi_{l,\infty}$ for $l\leq k$. 
Note that $W_\infty^\Pi$ 
can be identified with $(V_\infty^\sigma)^*$.

Now for $j\in I$ and $l\leq k$, 
we have $F_{k,l}e_j = e_j F_{k,l}$ 
and $F_{k,l} k_j = q_i^{2\delta_{i,j}(k - l)} k_j F_{k,l}$.
So $F_{k,l}^\sigma (k_j^{-1}f_j) = (k_j^{-1} f_j) F_{k,l}^\sigma$ 
and $F_{k,l}^\sigma k_j = q_i^{2\delta_{i,j}(k - l)} k_j F_{k,l}^\sigma$.
This implies 
$$
\Pi_{l,k} (f_jk_j^2) = (f_jk_j^2) \Pi_{l,k}
\quad
\text{ and }
\quad
\Pi_{l,k} k_j = q_i^{2\delta_{i,j}(k - l)} k_j \Pi_{l,k}\,.
$$

For $0\leq j\leq n$, we set $\tilde{e}_j = e_j k_j^{-1}$.

\begin{prop} 
Let $v\in W_\infty^\Pi$, $0\leq j\leq n$ and $l\geq 1$.
When $k\rightarrow \infty$,  
$(\Pi_{l,k} \tilde{e}_j \Pi_{k,\infty})(v)$ converges 
to a vector $v_{j,l}\in W_l$. 
Moreover we have $\Pi_{l,l'}(v_{j,l'}) = v_{j,l}$ for $l\leq l'$. 
\end{prop}

\begin{proof} 
First assume that 
$j\in I$, and let $v\in W_\infty^\Pi$.  
Using $S^{-1}\bigl(\tilde{e}_j\bigr) = -q_j^{-2}e_j$ 
in $U_{q^{-1}}(\mathfrak{g})$
and $\sigma(e_j) = q_j^2\tilde{x}_{j,0}^-$, 
we have for $k>l$
$$(\Pi_{l,k} \tilde{e}_j  \Pi_{k,\infty})(v) 
= (\tilde{e}_j \Pi_{k,\infty}(v))F_{k,l}^\sigma 
= - q_j^{-2}\Pi_{k,\infty}(v)\, e_j\, F_{k,l}^\sigma
= - v F_{\infty,k}\,\tilde{x}_{j,0}^-\, F_{k,l}\,.$$
The right hand side can be written as 
$- v F_{\infty,l+1} ((F_{l+1,k})^{-1}\tilde{x}_{j,0}^-\, F_{k,l})$.
This converges as $k\to\infty$ since 
$F_{k,l+1}^{-1}\tilde{x}_{j,0}^-F_{k,l}$ does.

For $k\geq l'\geq l$, we have
$$
(\Pi_{l,k} \tilde{e}_j  \Pi_{k,\infty})(v) 
= \Pi_{l,l'} ((\Pi_{l',k} \tilde{e}_j \Pi_{k,\infty})(v)).
$$
This implies the relation $\Pi_{l,l'}(v_{j,l'}) = v_{j,l}$.

We have $\sigma(S^{-1}(\tilde{e}_0)) 
= -q_0^{-2} \sigma(e_0) \in \CC\langle\tilde{x}_{j,m}^+\rangle_{j\in I ,m\in\ZZ}$.
As each 
$F_{k,l}^{-1}\tilde{x}_{j,m}^{+} F_{k,l}$ 
converges 
when $k\rightarrow +\infty$ in $\text{End}(V_l)$,
we can conclude as above.
\end{proof}

Let $\tilde{e}_j(v) \in W_\infty^\Pi$ be the projective limit of the 
$(v_{j,l})_{l\geq 1}$. We get a linear operator 
$\tilde{e}_j \in\text{End}(W_\infty^\Pi)$.
$W_\infty^\Pi$ has a natural $Q$-grading. 
We define the action of $k_j^{\pm 1}$
on $W_\infty^\Pi$ so that $w_\infty^\Pi$ has weight $1$. 
Then we get a structure of $U_q(\bo)$-module on $W_\infty^\Pi$. 

\begin{thm} $W_\infty^\Pi$ is irreducible isomorphic to $L_{i,1}^+$.
\end{thm}
\begin{proof} We have seen that each $v_k$ has $\ell$-weight
$N_k$. So the $\ell$-weight of 
$W_\infty^\Pi$ 
is the highest $\ell$-weight of $L_{i,1}^+$. In particular
$L_{i,1}^+$ is a subquotient of $W_\infty^\Pi$.
By construction, we have $\chi(W_\infty^\Pi) = \chi(V_\infty)$.
Now the result follows from Theorem \ref{plusexpl}.
\end{proof}

\bigskip

\noindent

{\it Acknowledgments.}\quad
\medskip

Research of MJ is supported by the Grant-in-Aid for Scientific 
Research B-23340039. 
Research of DH is supported partially by ANR through Projects
``G\'esaq" and ``Q-diff".

We wish to thank the organizers of the workshops
``Algebraic Lie Structures with Origins in Physics''
and 
``Discrete Integrable Systems'' 
at the Isaac Newton Institute in March 2009  
for giving us the opportunity 
for discussions, from which the present work was started.

DH thanks Vyjayanthi Chari for useful discussions and references.

\end{document}